\documentclass[11pt,letterpaper]{article}
\usepackage{mathrsfs}
\usepackage{geometry}
\usepackage{graphicx}
\usepackage{amssymb}
\usepackage{epstopdf}
\usepackage{amsmath,amscd}
\usepackage{amsthm}
\theoremstyle{plain} \textwidth=430pt \textheight=660pt

\makeatletter
\newcommand{\rmnum}[1]{\romannumeral #1}
\newcommand{\Rmnum}[1]{\expandafter\@slowroman #1@}
\makeatother

\newtheorem{theorem}{Theorem}[section]
\newtheorem{lemma}[theorem]{Lemma}
\newtheorem{corollary}[theorem]{Corollary}
\newtheorem{proposition}[theorem]{Proposition}

\newtheorem{example}[theorem]{Example}

\title{Spectral Curve of Periodic Fisher Graphs}
\author{Zhongyang Li\footnote{Department of Pure Mathematics and Mathematical Statistics, University of Cambridge,
Cambridge, UK, CB30WA, z.li@statslab.cam.ac.uk}}
\date{}
\begin{document}
\maketitle
\begin{abstract}
We study the spectral curves of dimer models on periodic Fisher
graphs, obtained from a ferromagnetic Ising model on $\mathbb{Z}^2$. The spectral curve is  defined by the zero locus of the determinant of a modified
weighted adjacency matrix. We prove that either they are disjoint
from the unit torus ($\mathbb{T}^2=\{(z,w):|z|=1,|w|=1\}$) or they
intersect $\mathbb{T}^2$ at a single real point.
\end{abstract}

\section{Introduction}
In this paper we study the spectral curve of periodic, 2-dimensional Fisher graphs, either finite in on direction, and periodic in the other direction (cylindrical graph); or bi-periodic and obtained from a bi-periodic, ferromagnetic Ising model on $\mathbb{Z}^2$ (toroidal graph). To an edge-weighed, cylindrical (resp. toroidal) Fisher graph, one associates its spectral curve $P(z)=0 (resp. P(z,w)=0)$. The real polynomial $P(z) (resp. P(z,w))$ defining the spectral curve arises as the characteristic polynomial of the Kasteleyn operator in the dimer model.

The study of spectral curve for periodic Fisher graphs (which is non-bipartite), is inspired by the work of Kenyon, Okounkov and Sheffield \cite{ko,kos}. They prove that the spectral curve of bipartite dimer models with positive edge weights is always a real curve of a special type, namely it is a Harnack curve. This implies many qualitative and quantitative results about the behavior of bipartite dimer models related to the phase transition.

The \textbf{Fisher graph} we consider in this paper is a graph with
each vertex of the honeycomb lattice replaced by a triangle, see
Figure 1.

\begin{figure}[htbp]
  \centering
\scalebox{0.8}[0.8]{\includegraphics{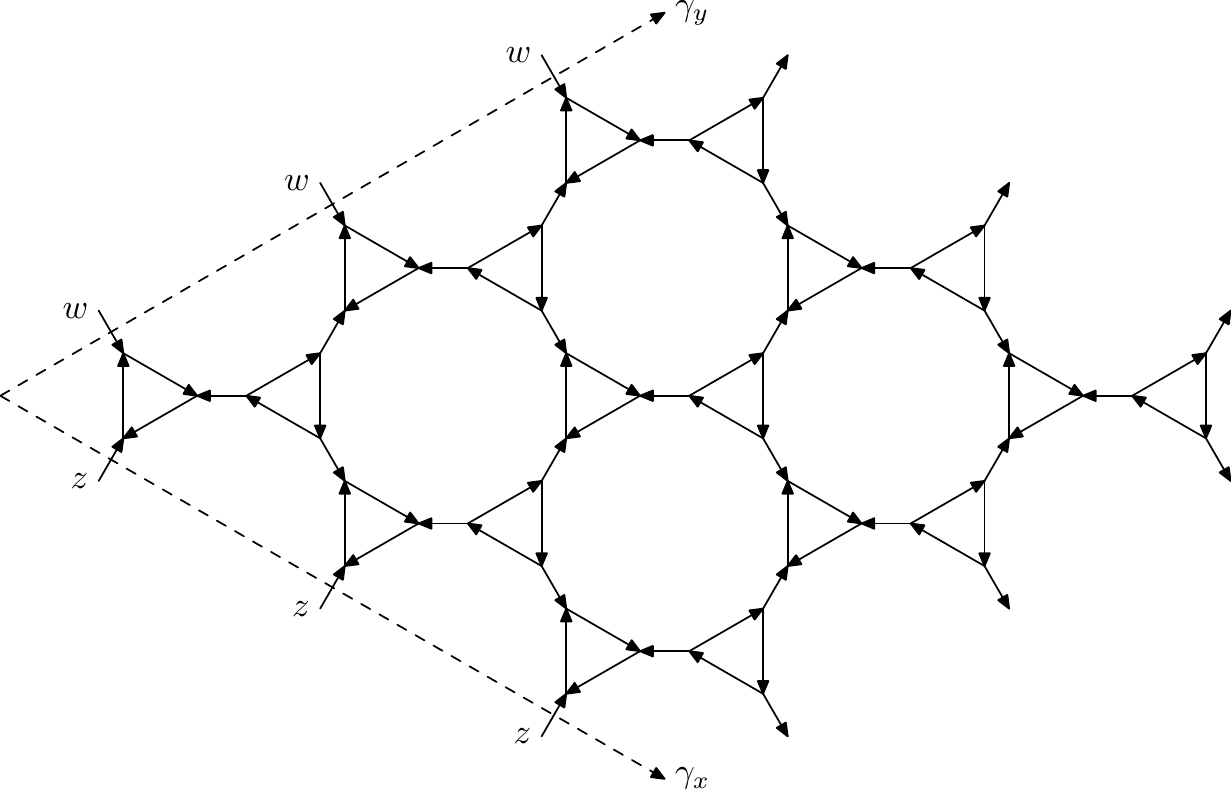}}
   \caption{Fisher graph }
\end{figure}

A planar graph is one which can be embedded into the plane such that edges can intersect only at vertices. Fix an embedding of a planar graph $G$. A
\textbf{clockwise-odd orientation} of $G$ is an orientation of the
edges such that for each face (except the outer face) an odd
number of edges pointing along it when traversed clockwise. For a
planar graph, such an orientation always exists \cite{ka2}. The
Kasteleyn matrix corresponding to such a graph is a
$|V(G)|\times|V(G)|$ skew-symmetric matrix $K$ defined by
\begin{align*}
K_{u,v}=\left\{\begin{array}{cc}W(uv)&{\rm if}\ u\sim v\ {\rm and}\
u\rightarrow v \\-W(uv)&{\rm if}\ u\sim v\ {\rm and}\ u\leftarrow
v\\0&{\rm else}.
\end{array}\right.
\end{align*}
where $W(uv)>0$ is the weight associated to the edge $uv$.

Now let $G$ be a $\mathbb{Z}^2$-periodic planar graph. By this we
mean $G$ is embedded in the plane so that translations in
$\mathbb{Z}^2$ act by weight-preserving isomorphisms of $G$ which
map each edge to an edge with the same weight. Let $G_n$ be the
quotient graph $G/(n\mathbb{Z}\times n\mathbb{Z})$. It is a finite
graph on a torus. Let $\gamma_{x,n}$(resp.\ $\gamma_{y,n}$) be a path in the
dual graph of $G_n$ winding once around the torus
horizontally(resp.\ vertically). Let $E_H$ (resp.\ $E_V$) be the set of edges crossed
by $\gamma_{x}$(resp.\ $\gamma_y$). We give a \textbf{crossing orientation}
for the toroidal graph $G_n$ as follows. We orient all the edges of
$G_n$ except for those in $E_H\cup E_V$. This is possible since no
other edges are crossing. Then we orient the edges of $E_H$ as if
$E_V$ did not exist. Again this is possible since $G-E_V$ is planar.
To complete the orientation, we also orient the edges of $E_V$ as if
$E_H$ did not exist.

Let $K_1$ be a Kasteleyn matrix for the graph $G_1$. Given any
parameters $z, w$, we construct a matrix $K(z,w)$ as follows. Let
$\gamma_{x,1}$ and $\gamma_{y,1}$ be the paths introduced as above.
Multiply $K_{u,v}$ by $z$ if the orientation on that edge is
from $u$ to $v$, and multiply $K_{u,v}$ by $\frac{1}{z}$ if the orientation is from $v$ to $u$, and
similarly for $w$ on $\gamma_y$. Define the \textbf{characteristic
polynomial} $P(z,w)=\det K(z,w)$. The \textbf{spectral curve} is
defined to be the locus $P(z,w)=0$.

We also discuss cylindrical graphs in this paper. Assume we have a
Fisher graph, periodic in the $z$ direction, finite in the $w$ direction, by
removing all $w$-edges and $\frac{1}{w}$ edges. An example of such a graph is shown in
Figure 2.

\begin{figure}[htbp]
  \centering
\scalebox{0.8}[0.8]{\includegraphics{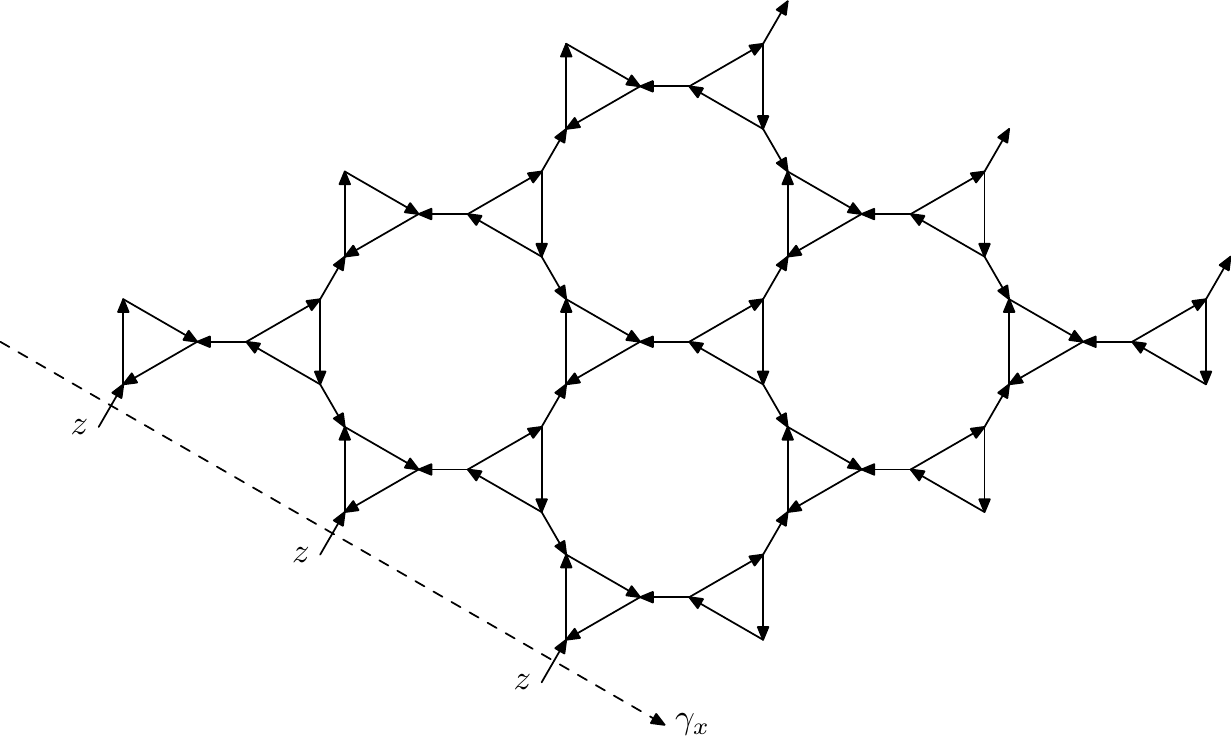}}
   \caption{Fisher graph on a cylinder}
\end{figure}

Assume it is $m\times n$, meaning that it has width $m$ with respect
to $z$, but period $n$ with respect to $w$. We embed this graph into
a cylinder. Let $K(z)$ be the corresponding weighted adjacency
matrix with an orientation shown in Figure 2. $K(z)$ is obtained
from $K(z,w)$ by giving all edges crossed by $\gamma_y$ weight 0. Let $P(z)=\det K(z)$.

Kenyon and Okounkov (\cite{ko}) proved that the spectral curves of
any periodic bipartite graph with positive edge weights are
\textbf{Harnack curves} \cite{mik,mr}, whose intersection with
$\mathbb{T}_{x,y}=\{(z,w):|z|=e^x,|w|=e^y\}$ can only be:
$\rmnum{1}$.~no intersection; $\rmnum{2}$.~a pair of conjugate
points, each of which is of multiplicity 1; $\rmnum{3}$.~a single
real zero of multiplicity 2 (real node). Since the Fisher graph we consider in this paper is
non-bipartite, no previous result is known. We prove the following
theorems

\begin{theorem}
Consider a positively-weighted cylindrical Fisher graph, which is periodic in one direction, and finite on the other direction. Its quotient graph under the translation can be embedded into a cylinder, as illustrated in Figure 2. The intersection of $P(z)=0$
and $\mathbb{T}$ is either empty or a single real point. \end{theorem}

In Figure 1, each edge of the graph is either a side of the triangle (triangular edge), or an edge connecting different triangles (non-triangle edge). We call all the horizontal non-triangle edges in Figure 1, $a$-edges.

\begin{theorem}
Consider  a bi-periodic Fisher
Graph, given edge-orientations as illustrated in Figure 1. Assume all the triangular edges and $a$-edges have weights 1, and all the other edges have weights in (0,1), then
\begin{enumerate}
\item the spectral curve $P(z,w)=0$ is a Harnack curve;
\item the only possible intersection of $P(z,w)=0$ with the unit torus $\mathbb{T}^2=\{(z,w):|z|=1,|w|=1\}$ is a single real point of multiplicity 2.
\end{enumerate}
\end{theorem}

Our result is very promising. Firstly, it leads to important
properties of the dimer model on cylindrical graphs, such as weak
convergence of Boltzmann measures and convergence rate of
correlations for an infinite, periodic graph with finite width along
one direction, see Proposition 3.4. Secondly, since the dimer model
on the Fisher graph are closely related to the Ising model \cite{mw}
and the vertex model \cite{lk}, our result leads to a quantitative characterization of the 
critical temperature of the arbitrary periodic ferromagnetic, two-dimensional Ising model, as the solution of an algebraic equation.  \cite{li}.

The outline of the paper is as follows.  In Section 2, we  explain the connection between the Kasteleyn operator, characteristic polynomial, spectral curve with the dimer model. In section 3, prove Theorem 1.1, as well as a phase transition characterized by the decay rate of edge-edge correlation for the dimer model on the cylindrical Fisher graph, resulting from Theorem 1.1. In section 4, we prove Theorem 1.2. The proof consists of 2 critical components, one is that $P(z,w)\geq 0$ for any $(z,w)\in\mathbb{T}^2$, the other is an explicit correspondence between the Kasteleyn operator on the Fisher graph, and the Kasteleyn operator in the square-octagon lattice, introduced in \cite{d}.
\\
\\
\textbf{Acknowledgements} It is a pleasure to express the gratitude
to Richard Kenyon for suggesting the problem and for helpful
discussions. The author would like to thank also David Cimasoni, Hugo Duminil-Copin and David Wilson for 
valuable comments, and C$\acute{e}$dric Boutillier for introducing reference \cite{d}.  The author acknowledges support from the EPSRC under grant EP/103372X/1.

\section{Background}

A \textbf{perfect matching}, or a \textbf{dimer cover}, of
a graph is a collection of edges with the property that each vertex
is incident to exactly one edge. A graph is \textbf{bipartite} if
the vertices can be 2-colored, that is, colored black and white so
that black vertices are adjacent only to white vertices and vice
versa.

To a weighted finite graph $G=(V,E,W)$, the weight $W:
E\rightarrow\mathbb{R}^{+}$ is a function from the set of edges to
positive real numbers. We define a probability measure, called the
\textbf{Boltzmann measure} $\mu$ with sample space the set of dimer
covers. Namely, for a dimer cover $D$
\begin{eqnarray*}
\mu(D)=\frac{1}{Z}\prod_{e\in D}W(e)
\end{eqnarray*}
where the product is over all edges present in $D$, and $Z$ is a
normalizing constant called the \textbf{partition function}, defined
to be
\begin{eqnarray*}
Z=\sum_{D}\prod_{e\in D}W(e),
\end{eqnarray*}
the sum over all dimer configurations of $G$.

If we change the weight function $W$ by multiplying the edge weights
of all edges incident to a single vertex $v$ by the same constant,
the probability measure defined above does not change. So we define
two weight functions $W,W'$ to be \textbf{gauge equivalent} if one
can be obtained from the other by a sequence of such
multiplications.

The key objects used to obtain explicit expressions for the dimer
model are \textbf{Kasteleyn matrices}. They are weighted, oriented
adjacency matrices of the graph $G$ defined as in Page 1. 

It is known \cite{ka1,ka2,tes,kos} that for a planar graph with a
clockwise odd orientation, the partition function of dimers
satisfies
\begin{align*}
Z=\sqrt{\det K}.
\end{align*}

Given a Fisher graph with an orientation as illustrated in Figure 1, the quotient graph can be embedded into an $n\times n$ torus. When $n$ is even, if we reverse the orientations of all the edges crossed by $\gamma_x$ and all the edges crossed by $\gamma_y$, the resulting orientation is a crossing orientation. For $\theta,\tau\in\{0,1\}$, given the orientation as in Figure 1, let $K_n^{\theta,\tau}$ be the
Kasteleyn matrix $K_n$ in which the weights of edges in $E_H$ are
multiplied by $(-1)^{\theta}$, and those in $E_V$ are multiplied by
$(-1)^{\tau}$. Using the result proved in \cite{tes}, we can derive that when $n$ is even, the partition
function $Z_n$ of the graph $G_n$ is
\begin{align*}
Z_n=\frac{1}{2}|-\mathrm{Pf}(K_n^{00})+\mathrm{Pf}(K_n^{10})+\mathrm{Pf}(K_n^{01})+\mathrm{Pf}(K_n^{11})|.
\end{align*}

Let $E_m=\{e_1=u_1v_1,...,e_m=u_mv_m\}$ be a subset of edges of
$G_n$. Kenyon \cite{ke2} proved that the probability of these edges
occurring in a dimer configuration of $G_n$ with respect to the
Boltzmann measure $P_n$ is
\begin{align*}
P_n(e_1,...,e_m)=\frac{\prod_{i=1}^{m}W(u_iv_i)}{2Z_n}|-\mathrm{Pf}(K_n^{00})_{E_m^{c}}+\mathrm{Pf}(K_n^{10})_{E_m^{c}}+\mathrm{Pf}(K_n^{01})_{E_m^{c}}+\mathrm{Pf}(K_n^{11})_{E_m^{c}}|
\end{align*}
where $E_m^c=V(G_n)\setminus\{u_1,v_1,...,u_m,v_m\}$, and
$(K_n^{\theta\tau})_{E_m^c}$ is the submatrix of $K_n^{\theta\tau}$
whose lines and columns are indexed by $E_m^c$.

The asymptotic behavior of $Z_n$ when $n$ is large is an interesting
subject. One important concept is the partition function per
fundamental domain, which is defined to be
\begin{eqnarray*}
\lim_{n\rightarrow\infty}\frac{1}{n^2}\log Z_n.
\end{eqnarray*}

Gauge equivalent dimer weights give the same spectral curve. That is
because after Gauge transformation, the determinant is scaled by a
nonzero constant, thus not changing the locus of $P(z,w)$.

A formula for enlarging the fundamental domain is proved in
\cite{ckp,kos}. Let $P_n(z,w)$ be the characteristic polynomial of
$G_n$, and $P_1(z,w)$ be the characteristic polynomial of $G_1$,
then
\begin{eqnarray*}
P_n(z,w)=\prod_{u^n=z}\prod_{v^n=w}P_1(u,v)
\end{eqnarray*}

\section{Graph on a Cylinder}

\subsection{Spectral Curve}

\textbf{Proof of Theorem 1.1} Without loss of generality, assume the
period $n$, or the circumference of the cylinder, is even. Assume
$P(z)=0$ has a non-real zeros $z_0\in\mathbb{T}^2$. Assume
\begin{eqnarray*}
z_0=e^{i\alpha_0\pi},\qquad \alpha_0\in(0,1)\cup(1,2)
\end{eqnarray*}
We classify all the real numbers in $(0,1)\cup(1,2)$ into 3 types
\begin{enumerate}
\item $\alpha_0=\frac{p}{q}$ where $p,q$
 are positive integers with no common factors, $p$ is odd
 \item $\alpha_0$
is irrational
\item $\alpha_0=\frac{p}{q}$ where $p,q$
 are positive integers with no common factors, $p$ is even
 \end{enumerate}
 First let us consider Case 1 and Case 2. There exists a sequence
 $\ell_k\in\mathbb{N}$, such that
 \begin{eqnarray*}
 \lim_{k\rightarrow\infty}z_0^{\ell_k}=-1
 \end{eqnarray*}
 In other words, if we assume
 $z_0^{\ell_k}=e^{\sqrt{-1}\alpha_k\pi}$ where
 $\alpha_k\in[0,2)$, then
 \begin{eqnarray*}
 \lim_{k\rightarrow\infty}\alpha_k=1
 \end{eqnarray*}
 According to the formula of enlarging the fundamental domain,
 \[P(z_0)=0\]
 implies
 \[P_{\ell_k}(z_0^{\ell_k}))=0\qquad \forall k\]
 Since the cylindrical graph is actually planar, if we reverse the orientation of all the edges crossed by $\gamma_x$, we get a clockwise-odd orientation, given that $n$ is even. Hence
 \[P_{\ell_k}(-1)=Z_{\ell_k}^2\]
 where $Z_{\ell_k}$ is the partition function of dimer configurations of the cylinder with circumference $n\ell_k$ and height $m$. Therefore we have
 \begin{eqnarray}
 1&=&\lim_{k\rightarrow\infty}\left|\frac{P_{\ell_k}(z_0^{\ell_k})-P_{\ell_k}(-1)}{P_{\ell_k}(-1)}\right|\\
 &=&\lim_{k\rightarrow\infty}\frac{1}{Z_{\ell_k}^2}|P_{\ell_k}(e^{i\alpha_k\pi})-P_{\ell_k}(e^{i\pi})|\\
 &=&\lim_{k\rightarrow\infty}\frac{1}{Z_{\ell_k}^2}\left|\sum_{t=1}^{\infty}\frac{[\pi(\alpha_k-1)]^{2t}}{(2t)!}\frac{\partial^{2t}P_{\ell_k}(e^{i\theta})}{\partial\theta^{2t}}|_{\theta=\pi}\right|\label{equaltoone}
 \end{eqnarray}
Since
\begin{eqnarray*}
P_{\ell_k}(z)=\sum_{0\leq j\leq m}P_{j}^{(\ell_k)}(z_j+\frac{1}{z^j})
\end{eqnarray*}
where $P_j^{(\ell_k)}$ is the signed sum of loop configurations winding exactly $j$ times around the cylinder (see Lemma 2.1), we have
\begin{eqnarray*}
\frac{\partial^{2t}P_{\ell_k}(e^{i\theta})}{\partial\theta^{2t}}|_{\theta=\pi}=\sum_{1\leq j\leq m}2j^{2t}(-1)^tP_j^{(\ell_k)},
\end{eqnarray*}
and
\begin{eqnarray*}
\lim_{k\rightarrow\infty}\frac{1}{Z_{\ell_k}^2}\left|\sum_{t=1}^{\infty}\frac{[\pi(\alpha_k-1)^{2t}]}{(2t)!}\frac{\partial^{2t}P_{\ell_k}(e^{i\theta})}{\partial\theta^{2t}}|_{\theta=\pi}\right|\\
\leq\lim_{k\rightarrow\infty}\frac{1}{Z_{\ell_k}^2}\sum_{t=1}^{\infty}\frac{[\pi m(\alpha_k-1)]^{2t}}{(2t)!}2\sum_{1\leq j\leq m}|P_j^{(\ell_k)}|
\end{eqnarray*}
To estimate $|P_j^{(\ell_k)}|$, let us divide the $m\times \ell_k n$ cylinder into $\ell_k n$, $m\times 1$ layers. Connecting edges between two layers may be occupied once, unoccupied, or occupied twice(appear as doubled edges). Choose one layer $L_0$, we construct an equivalent class of loop configurations. Two loop configurations are equivalent if they differ from each other only in $L_0$, and coincide on all the other layers and boundary edges of $L_0$.

We claim that for any given equivalent class, there is at least one configuration including only even loops. To see that we choose an arbitrary configuration in that equivalent class including odd loops. Let $\mathcal{T}$ be the set of all triangles in $L_0$ belonging to some loop crossing $L_0$. We choose an odd loop $s_1$. Choose an triangle $\Delta_1\in s_1\cap L_0$. Moving along $L_0$ until we find a triangle $\Delta_2$ belonging to a different non-planar odd loop $s_2$. This is always possible because by Lemma 2.2, any loop configuration including odd loops always has an even number of odd loops, and all the odd loops  are winding once along the cylinder. Change path through $\Delta_1$. Starting from $\Delta_1$, we change the doubled edge configuration to alternating edges along a path in $L_0$, and change the path through all triangles, belonging to even loops, or $s_1$, between $\Delta_1$ and $\Delta_2$, then we change the path through $\Delta_2$. For any even loop between $\Delta_2$ and $\Delta_2$, we change paths through an even number of triangles of that even loop, hence it is  still an even loop. However, for $s_1$ and $s_2$, we change paths through an odd number of triangles, then both of them become even. We can continue this process until we eliminate all the odd loops, because there are always an even number of odd loops in the configuration. An example of such an path change process is illustrated in the following figures.

\begin{figure}[htbp]
  \centering
\scalebox{0.8}[0.8]{\includegraphics{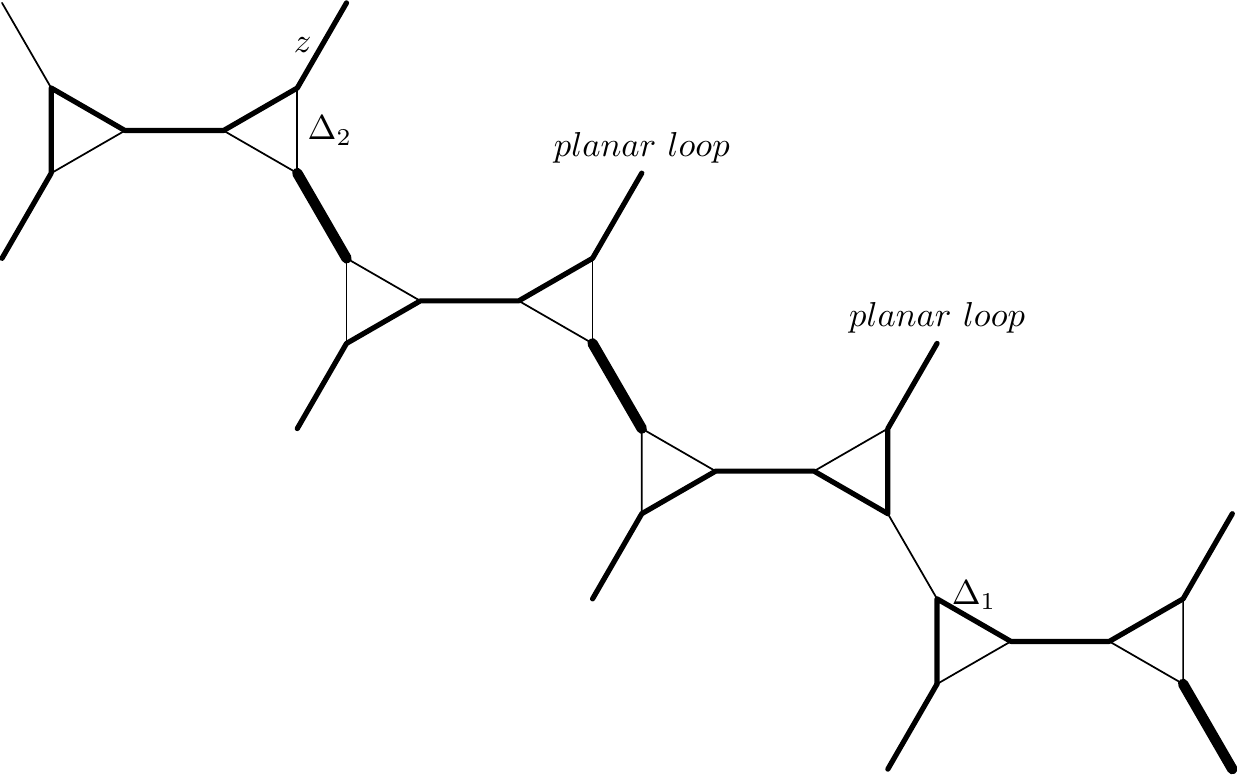}}
   \caption{before path change}
\end{figure}

\begin{figure}[htbp]
  \centering
\scalebox{0.8}[0.8]{\includegraphics{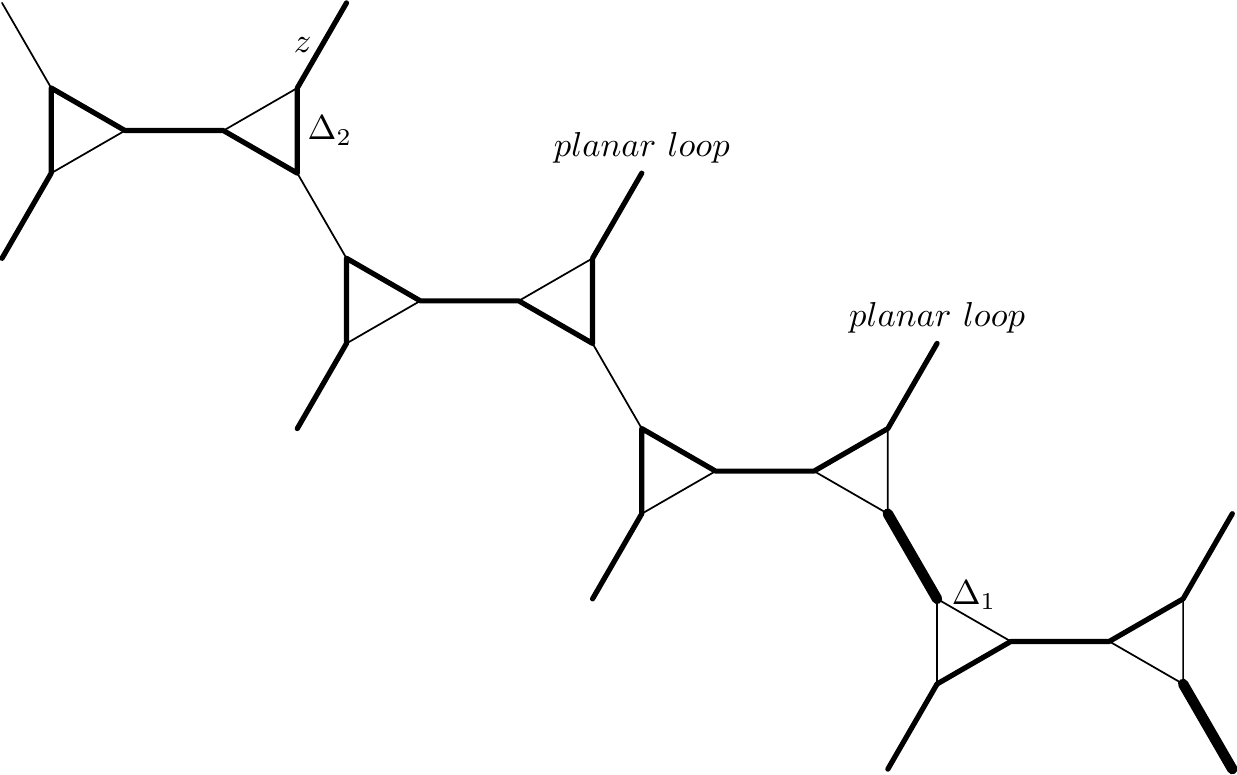}}
   \caption{after path change}
\end{figure}

Since we have at most $2^m$ different configurations in each equivalent class, and each equivalent class has at least one configuration including only even loops, we have
\begin{eqnarray*}
\frac{\#\ \mathrm{of\ loop\ configurations\ including\ odd\ loops}}{\#\ \mathrm{of\ even\ loop\ configurations}}<2^m
\end{eqnarray*}
We have a finite number of different edge weights, each of which is positive, hence the quotient of any two weights is bounded by a constatn $C_1$, then
\begin{eqnarray*}
\sum_{1\leq j\leq m}|P_j^{(\ell_k)}|\leq \mathrm{Partition\ of\ configurations\ including\ nonplanar\ odd\ loops}\\+\mathrm{Partition\ of\ configurations\ with\ only\ even\ loops}\\
\leq(2^mC_1^{6m}+1)\mathrm{Partition\ of\ even\ loop\ configurations}\leq C_2^mZ_{\ell_k}^2
\end{eqnarray*}
As a result
\begin{eqnarray*}
\lim_{k\rightarrow\infty}\frac{1}{Z_{\ell_k}^2}\left|\sum_{t=1}^{\infty}\frac{[\pi(\alpha_k-1)]^{2t}}{(2t)!}\frac{\partial^{2t}P_k(e^{i\theta})}{\partial\theta^{2t}}\right|\\
\leq\lim_{k\rightarrow\infty}\frac{1}{Z_{\ell_k}^2}\sum_{t=1}^{\infty}\frac{[m\pi(\alpha_k-1)]^{2t}}{(2t)!}C_2^mZ_{\ell_k}^2
\end{eqnarray*}
Since $m$ is a constant, and $\lim_{k\rightarrow\infty}\alpha_k=1$, we have
\begin{eqnarray*}
\lim_{k\rightarrow\infty}\sum_{t=1}^{\infty}\frac{[m\pi(\alpha_k-1)]^{2t}}{(2t)!}C_2^m=0
\end{eqnarray*}
which is a contradiction to (\ref{equaltoone}). Hence for all $\alpha_0$ in Case 1 and Case 2, $P(e^{i\alpha_0\pi})\neq 0$.

Now let us consider $\alpha_0$ in Case 3. From the argument above we derive that as long as 
\[|\alpha-1|<\delta,\]
where $\delta$ is a small positive number depending only on $m$,
\[P(e^{i\alpha\pi})\neq 0.\]
Consider $\alpha_0=\frac{p}{q}$, where $p$, $q$ have no common factors and $p$ is even. We claim that as long as the denominator $1$ is sufficiently large, $P(z_0)\neq 0$. To see that, if
\[\frac{1}{q}<\frac{\delta}{2},\]
there exists an integer $k$, such that
\[\left|\frac{1}{\pi}Arg[e^{i k\alpha_0\pi}]-1\right|<\delta.\]
Hence after enlarging the fundamental domain to an $m\times kn$ cylinder, we have
\[P_k(e^{ik\alpha_0\pi})\neq 0,\]
hence
\[P(e^{i\alpha_0\pi})\neq 0.\]
Now we consider $\frac{1}{q}\geq \frac{\delta}{2}$, only finitely may $q$'s satisfy this condition. Let $\ell$ be a prime number satisfying $\ell>\left[\frac{2}{\delta}\right]+1$, then after enlarging the fundamental domain to an $m\times\ell n$ cylinder, the corresponding $\delta$ will not change because it depends only on $m$. For any
\[1\geq \frac{1}{q}\geq\frac{\delta}{2},\]
we have
\[\frac{\delta}{2}>\frac{1}{\ell q}.\]
Let $\ell\alpha_1(\mod 2)=\alpha_0$. Namely, 
\[\alpha_1=\frac{p+2sq}{\ell q}\qquad for\ s=1,2,...,\ell\]
in reduced form. By the previous argument $P(e^{\alpha_1i\pi})\neq 0$,  hence $P_{\ell}(e^{\alpha_0\pi\sqrt{-1}})\neq 0$. Hence when $\alpha_0$ is in Case 3, $P_{\ell}(z_0)\neq 0$ after enlarging the fundamental domain to $m\times n\ell$, where $\ell$ depends only on $m$, and is independent of $\alpha_0$. If $P(z_0)=0$, $z_0$ is not real, then we enlarge the fundamental domain to $m\times n\ell$, where $\ell$ is a big prime number depending only on $m$, we derive that $P_{\ell}(z_0^{\ell})=0$, however this is impossible since $P_{\ell}(z)=0$ can have only real root on $\mathbb{T}$.
  $\Box$

\begin{corollary}For any non-real $z\in\mathbb{T}$, all eigenvalues of $K(z)$ are of
the form $i\lambda_{j},j=1,...,6mn$, where $i$ is the imaginary
unit. $3mn$ of the $\lambda_{j}$'s are positive, the other
$\lambda_{j}$'s are negative.
\end{corollary}
\begin{proof} From the definition of $K(z)$,  $iK(z)$ is a Hermitian matrix.
We claim that $iK(z)$ has $3mn$ positive and $3mn$ negative
eigenvalues for any non-real $z\in\mathbb{T}$. In fact, for a planar
graph with no $z$ vertices, $K_0$ is a anti-symmetric real matrix
with eigenvalues $\pm i\lambda_{j}$, $1\leq j\leq 3mn-m$, and $\det
K_{0}> 0$, because $|Pf K_0|$ is the partition function of dimer
configurations with positive edge weights, and $\det K_{0}=(Pf
K_{0})^2$  By the previous theorem, every time we add a pair of
boundary vertices connected by a $z$ edge, we have a anti-Hermitian
matrix $K_{r+1}$ of order $6mn-2m+2(r+1)$ with $K_r$ as a principal
minor. By induction hypothesis $iK_r$ has the same number of
positive and negative eigenvalues, the interlacing theorem implies
that $iK_{r+1}$ has at least $3mn-m+r$ positive and $3mn-m+r$
negative eigenvalues. By previous theorem, $\det K_{r+1}>0$ for
non-real $z$, so the other two eigenvalues of $iK_{r+1}$ can only be
one positive and one negative.
\end{proof}

\begin{lemma} If $P(1)=0$, then $\left.\frac{\partial P(z)}{\partial z}\right|_{z=1}=0$. Let $z=e^{\sqrt{-1}\theta}$, then $\left.\frac{\partial P(e^{\sqrt{-1}\theta})}{\partial \theta}\right|=0$. For generic choice of edge weights, $\left.\frac{\partial^2 P(z)}{\partial z^2}\right|_{z=1}\neq 0$, $\left.\frac{\partial^2 P(e^{\sqrt{-1}\theta})}{\partial \theta^2}\right|\neq 0$
\end{lemma}
\begin{proof}
We prove the result for derivatives with respect to $z$, the derivatives with respect to $\theta$ are very similar. Since
\begin{eqnarray*}
P(z)=\sum_{j}P_j(z^j+\frac{1}{z^j})
\end{eqnarray*}
we have
\begin{eqnarray*}
\left.\frac{\partial P}{\partial z}\right|_{z=1}&=&\sum_{j}j \left.P_j(z^{j-1}-\frac{1}{z^{j+1}})\right|_{z=1}=0\\
\left.\frac{\partial^2 P}{\partial z^2}\right|_{z=1}&=&\sum_{j}2j^2P_j
\end{eqnarray*}
Each monomial in $P(1)$ corresponds to a loop configuration including only even loops. Consider a configuration including $2k$ even loops, each of which winding exactly once around the cylinder. Except those non-contractible even loops, the rest of the graph is covered by doubled-edge configuration. Consider $P(1)$ as a polynomial of edge weights. The monomial corresponding to that configuration has a coefficient
\[S_{0,k}=\pm 2^{2k}\]
because each single loop can have two  different orientations, corresponding to two terms in the expansion of the determinant. However, the coefficient of the monomial corresponding to the same configuration in $\left.\frac{\partial^2 P(z)}{\partial z^2}\right)_{z=1}$ is 
\[S_{2,k}=(2k)^2+2k(2k-2)^2+\left(\begin{array}{c}2k\\2\end{array}\right)(2k-4)^2+\cdots+\left(\begin{array}{c}2k\\2k\end{array}\right)(2k-4k)^2\]
Obviously, $\frac{S_{2,k}}{S_{0,k}}$ is a number depending on $k$. Therefor $P(1)$ and $\left.\frac{\partial^2 P(z)}{\partial z^2}\right|_{z=1}$ cannot divide each other. Let
\begin{eqnarray*}
W_2&=&\{(w_e)_{e\in E}:\left.\frac{\partial^2 P(z)}{\partial z^2}\right|_{z=1}=0\}\\
W_0&=&\{(w_e)_{e\in E}: P(1)=0\}
\end{eqnarray*}
The intersection of $W_2$ and $W_0$ forms a proper subvariety of $W$. Hence if $P(1)=0$, and we choose the edge weights generically, $\left.\frac{\partial^2 P(z)}{\partial z^2}\right|_{z=1}\neq 0$.
\end{proof}
\noindent
\textbf{Remark.} Lemma 3.3 implies that if $P(1)=0$, for generic choice of edge weights, the intersection is of multiplicity 2.

\begin{proposition} Let 
\[F(r)=\frac{1}{2\pi}\int_{0}^{2\pi}\log P(re^{i\theta})d\theta\]
If $P(z)=0$ does not intersect the unit torus $\mathbb{T}$, $F(r)$ is differentiable at $r=1$; If $P(z)=0$ has a real zero of multiplicity 2 at $z=1$, 
\[\lim_{r\rightarrow 1+}\frac{\partial F(r)}{\partial r}-\lim_{r\rightarrow 1-}\frac{\partial F(r)}{\partial r}=2\]
\end{proposition}
\begin{proof} Since $P(z)$ is a Laurent polynomial in $z$, according to the Jensen's formula
\[\log|P(0)|=-\sum_{k=1}^{n}\log\left(\frac{r}{|a_k|}\right)+\frac{1}{2\pi}\int_{0}^{2\pi}\log|P(re^{i\theta})|d\theta\]
where $a_1,\cdots,a_n$ are the zeros of $P$ in the interior of the disk $\{z:|z|<r\}$. If $P(z)$ has no zeros on the circle $\{z:|z|=r\}$, then
\begin{eqnarray*}
\frac{\partial F}{\partial r}&=&\lim_{\Delta r\rightarrow 0}\frac{F(r+\Delta r)-F(r)}{\Delta r}\\
&=&\lim_{\Delta r\rightarrow 0}\frac{1}{\Delta r}\sum_{k=1}^{n}\log\left(\frac{r+\Delta r}{r}\right)\\
&=&\frac{n}{r}
\end{eqnarray*}
The proposition follows from substituting $r=1$, and the fact that the intersection of $P(z)=0$ with $\mathbb{T}$ can only be a single real point.
\end{proof}

\noindent\textbf{Remark.} Jensen's formula implies that the curve $P(z)=0$ is Harnack if and only if 
\[r_0(\lim_{r\rightarrow{r_0+}}-\lim_{r\rightarrow{r_0-}})\frac{\partial F(r)}{\partial r}\leq 2,\] for any $r_0>0$. In fact, if the height of the cylinder $m\leq 3$, we can always derive that the corresponding spectral curve is Harnack. To see that, first of all, being Harnack is a closed condition, for generic choice of edge weights, the intersection of $P(z)=0$ with $|z|=1$ is at most two points(counting multiplities). Without loss of generality, assume for some $r(0<r<1)$, $P(z)=0$ intersects $|z|=r$ at 3 different points $z_1,z_2,z_3$, then $\bar{z}_1,\bar{z}_2,\bar{z}_3$, lie also on the intersection of $P(z)=0$ and $|z|=r$, given that $P(z)$ is a real-coeffient polynomial.  Hence the intersection of $P(z)=0$ with $|z|=r$ is at least 4 points. Moreover, by symmetry, the reciprocal of those points are also roots of $P(z)=0$, then $P(z)=0$ has at least 8 roots, which is a contradiction to the fact $m\leq 3$.

\subsection{Limit Measure of Cylindrical Approximation}

From the proof we know that all terms in $P_{m\times n}(-1)$ are
positive, if $n$ is even. We can always enlarging the fundamental
domain in $z$ direction without changing edge weights to get
\begin{eqnarray*}
P_{m\times 2n}(-1)=P_{m\times n}(i)P_{m\times n}(-i)=P_{m\times
n}^2(i)=|Pf_{m\times 2n}K(-1)|^2
\end{eqnarray*}
Therefore $P_{m\times n}(i)$ is the partition function of dimer
configurations of the $m\times 2n$ cylinder graph. According to the
formula of enlarging the fundamental domain,
\[P_{m\times 2ln}(-1)=\prod_{z^{2l}=-1}P_{m\times n}(z),\]
we have
\begin{eqnarray}
\lim_{l\rightarrow\infty}\frac{1}{4l}\log P_{m\times 2ln}(-1)&=&\lim_{l\rightarrow\infty}\frac{1}{4\pi}\frac{2\pi}{2l}\sum_{z^{2l}=-1}\log P_{m\times n}(z)\\
&=&\frac{1}{4\pi}\int_{\mathbb{T}}\log P_{m\times
n}(z)\frac{dz}{iz}\label{freeenergy}
\end{eqnarray}
The convergence of the Riemann sums to the integral follows from the
fact that the only possible zeros of $P(z)$ on $\mathbb{T}$ is a
single real node. (\ref{freeenergy}) is defined to be the
\textbf{partition function per fundamental domain}.

Any probability measure on the infinite banded graph, with depth $m$ on one direction, and period $n$ on the other direction, is determined by the probability of cylindrical sets. Namely, we choose a finite number of edges $e_1,e_2,\cdots, e_k$ arbitrarily, and the probabilities
\[Pr(e_1\&e_2\&\cdots\&e_k)\]
that $e_1,e_2,\cdots,e_k$ occur in the dimer configuration simultaneously for all finite edge sets determines the probability measure. We consider the measures on the infinite graph as weak limits of measures on cylindrical graphs.  First of all, we prove a lemma about the entries of the inverse Kasteleyn matrix using the cylindrical approximation.
\begin{lemma}
\begin{eqnarray*}
\lim_{l\rightarrow\infty} K^{-1}_{m\times
2ln}(-1)_{(k_v,s_v),(k_w,s_w)}=\frac{1}{2\pi}p.v.\int_{\mathbb{T}}z^{k_v-k_w}\frac{cofactor
K_{m\times n}(s_v,s_w)(z)}{P_{m\times n}(z)}\frac{dz}{iz}\end{eqnarray*}
\end{lemma}
\begin{proof}
To that end, we construct a transition matrix $S$
to make $S^{-1}K_{m\times 2ln}S$ block diagonal, with each block
corresponding to a $m\times n$ quotient graph. Define
\begin{eqnarray*}
S=(e_0^1,...,e_0^{6mn},e_1^1,...,e_1^{6mn},...,e_{2l-1}^1,...,e_{2l-1}^{6mn})
\end{eqnarray*}
where
\begin{eqnarray*}
e_k^s(j,t)=
\begin{cases} e^{\frac{i\pi(2j+1)(2k+1)}{4l}}&s=t\\0&s\neq t
\end{cases}
\end{eqnarray*}
then
\begin{eqnarray*}
S^{-1}K_{m\times 2ln}S=\left(\begin{array}{cccc}K_{m\times
n}(e^{\frac{i\pi}{2l}})&
&\multicolumn{2}{c}{\raisebox{1.3ex}[0pt]{\Huge0}} \\ &K_{m\times
n}(e^{\frac{3i\pi}{2l}})& & \\ & &\ddots&
\\ \multicolumn{2}{c}{\raisebox{1.3ex}[0pt]{\Huge0}}& &K_{m\times n}(e^{\frac{i(4l-1)\pi}{2l}})\end{array}\right)
\end{eqnarray*}
Since $S^{-1}=\frac{1}{2l}\bar{S}^t$, we have
\begin{eqnarray*}
K^{-1}_{m\times
2ln}(-1)_{(k_v,s_v),(k_w,s_w)}=\frac{1}{2l}\sum_{j=0}^{2l-1}\frac{cofactor
K_{m\times n}(s_v,s_w)}{P_{m\times
n}(e^{\frac{(2j+1)i\pi}{2l}})}e^{\frac{i(2j+1)\pi(k_v-k_w)}{2l}}
\end{eqnarray*}
where $k_v$ is the index of the fundamental domain for vertex $v$
and $s_v$ is the index of vertex in the fundamental domain.

If $P_{m\times n}(z)$ has no zero on $\mathbb{T}$, we have
\begin{eqnarray}
\lim_{l\longrightarrow\infty}K^{-1}_{m\times
2ln}(-1)_{(k_v,s_v),(k_w,s_w)}=\frac{1}{2\pi}\int_{\mathbb{T}}z^{k_v-k_w}\frac{cofactor
K_{m\times n}(s_v,s_w)(z)}{P_{m\times n}(z)}\frac{dz}{iz}\label{noncritical}
\end{eqnarray}

If 1 is an order-2 zero of $P_{m\times n}(z)$, let
\begin{eqnarray*}
Q(z)=z^{k_v-k_w}\frac{cofactor K_{m\times n}(s_v,s_w)(z)}{P_{m\times
n}(z)},
\end{eqnarray*}
Since $\det K(1)$ is an anti-symmetric real matrix of even order, and non-invertible, the dimension of its null space is non-zero and even. Hence $Adj K(1)$ is a zero matrix, and 1 is at least a zero of order 1 for $cofactor K_{m\times n}(s_v,s_w)(z)$. Then in a neighborhood of 1,
\begin{eqnarray*}
Q(z)=\frac{Res_{z=1}Q(z)}{z-1}+R(z),
\end{eqnarray*}
where $R(z)$ is analytic at 1.
\begin{eqnarray}
\lim_{l \rightarrow\infty}K^{-1}_{m\times
2ln}(-1)_{(k_v,s_v),(k_w,s_w)}=\lim_{\delta\rightarrow
0+}\lim_{l\rightarrow\infty}\frac{1}{2l}\left(\sum_{0\leq
j<\frac{l\delta}{\pi}-\frac{1}{2}}+\sum_{\frac{l\delta}{\pi}-\frac{1}{2}\leq
j\leq
2l-\frac{1}{2}-\frac{l\delta}{\pi}}+\sum_{2l-\frac{1}{2}-\frac{l\delta}{\pi}<j\leq
2l-1}\right)Q(e^{\frac{(2j+1)i\pi}{2l}})\label{expression}
\end{eqnarray}

For the second term,
\begin{eqnarray}
&&\lim_{\delta\rightarrow 0+}\lim_{l\rightarrow
0}\frac{1}{2l}\sum_{\frac{l\delta}{\pi}-\frac{1}{2}\leq j\leq
2l-\frac{1}{2}-\frac{l\delta}{\pi}}Q(e^{\frac{(2j+1)i\pi}{2l}})\\&=&\lim_{\delta\rightarrow
0+}\frac{1}{2\pi}\int_{\delta}^{2\pi-\delta}Q(e^{i\theta})d\theta\\
&=&p.v.\frac{1}{2\pi}\int_{\mathbb{T}}Q(z)\frac{dz}{iz}\\
&=&\frac{1}{2}Res_{z=1}Q(z)+\sum Res_{|z|<1}\frac{Q(z)}{z}\label{secondterm}
\end{eqnarray}
For the first and the third term
\begin{eqnarray*}
&&\lim_{\delta\rightarrow
0+}\lim_{l\rightarrow\infty}\frac{1}{2l}\left(\sum_{0\leq
j<\frac{l\delta}{\pi}-\frac{1}{2}}+\sum_{2l-\frac{1}{2}-\frac{l\delta}{\pi}<j\leq
2l-1}\right)Q(e^{\frac{(2j+1)i\pi}{2l}})\\&=&
\lim_{\delta\rightarrow
0+}\lim_{l\rightarrow\infty}\frac{1}{2l}\left(\sum_{0\leq
j<\frac{l\delta}{\pi}-\frac{1}{2}}+\sum_{2l-\frac{1}{2}-\frac{l\delta}{\pi}<j\leq
2l-1}\right)\left(\frac{Res_{z=1}Q(z)}{e^{\frac{(2j+1)i\pi}{2l}}-1}+R(e^{\frac{(2j+1)i\pi}{2l}})
\right)
\end{eqnarray*}
Since $R(z)$ is analytic in a neighborhood of 1,
\begin{eqnarray}
\lim_{\delta\rightarrow
0+}\lim_{l\rightarrow\infty}\frac{1}{2l}\left(\sum_{0\leq
j<\frac{l\delta}{\pi}-\frac{1}{2}}+\sum_{2l-\frac{1}{2}-\frac{l\delta}{\pi}<j\leq
2l-1}\right)R(e^{\frac{(2j+1)i\pi}{2l}})=0\label{principal}
\end{eqnarray}
Moreover,
\begin{eqnarray}
&&\lim_{\delta\rightarrow
0+}\lim_{l\rightarrow\infty}\frac{1}{2l}Res_{z=1}Q(z)\left(\sum_{0\leq
j<\frac{l\delta}{\pi}-\frac{1}{2}}+\sum_{2l-\frac{1}{2}-\frac{l\delta}{\pi}<j\leq
2l-1}\right)\frac{1}{e^{\frac{(2j+1)i\pi}{2l}}-1}\\
&=&\lim_{\delta\rightarrow
0+}\lim_{l\rightarrow\infty}\frac{1}{2l}Res_{z=1}Q(z)\sum_{0\leq
j<\frac{l\delta}{\pi}-\frac{1}{2}}\left(\frac{1}{e^{\frac{(2j+1)i\pi}{2l}}-1}+\frac{1}{e^{-\frac{(2j+1)i\pi}{2l}}-1}\right)\\
&=&-\lim_{\delta\rightarrow
0+}\frac{1}{2l}Res_{z=1}Q(z)\left[\frac{l\delta}{\pi}+\frac{1}{2}\right]=0\label{minor}
\end{eqnarray}
And the theorem follows from (\ref{noncritical}), (\ref{expression}), (\ref{secondterm}),(\ref{principal}), (\ref{minor}).
\end{proof}

\begin{proposition}Using a large cylinder to approximate the infinite periodic banded graph, we derive that the weak limit of probability measures of the dimer model exists. The probability of a cylindrical set under this limit measure is 
\begin{eqnarray*}
Pr(e_1\& e_2\&\cdots\&e_k)=\lim_{l\rightarrow\infty}\prod_{j=1}^{k}w_{e_j}\sqrt{\det K^{-1}_{m\times 2l n}\left(\begin{array}{ccccc}u_{1}&v_{1}&\cdots&u_{k}&v_{k}\\ u_1&v_1&\cdots&u_k&v_{k}\end{array}\right)}
\end{eqnarray*}
where $w_{e_j}$ is the weight of $e_j$, and $u_j,v_j$ are the two vertices of $e_j$.
\end{proposition}
\begin{proof}
On a finite cylindrical graph of $m\times 2ln$
\begin{eqnarray*}
Pr(e_1\& e_2\&\cdots\&e_k)=\frac{Z_{e_1,\cdots, e_k}}{Z}
\end{eqnarray*}
where $Z$ is the partition function of dimer configurations on that graph, and $Z_{e_1,\cdots, e_k}$ is the partition function of dimer configurations for which $e_1,\cdots e_k$ appear simultaneously. Since
\begin{eqnarray}
Z^2&=&\det K_{m\times 2ln}(-1)\\
Z^2_{e_1,\cdots,e_k}&=&\prod_{j=1}^{k}w_{e_j}cofactor K_{m\times 2ln,(u_1,v_1,\cdots,u_k,v_k)}(-1)\label{specialpartition}
\end{eqnarray}
(\ref{specialpartition}) follows from the fact that if originally we have a clockwise-odd orientation, we still have a clockwise-odd orientation when removing edges and ending vertices, while keeping the orientation on the rest of the graph. The proposition follows from Jacobi's formula for the determinant of minor matrices.
\end{proof}

We consider two edges $e_1$ and $e_2$ with weight $x_1$ and $x_2$ on
$m\times 2ln$ cylinder, and compute the covariance.
\begin{eqnarray}
&&Pr_{m\times 2ln}(e_1 \& e_2)-Pr_{m\times
2ln}(e_1)Pr_{m\times 2ln}(e_2)\\
&=&x_1x_2\sqrt{\det K^{-1}_{m\times 2ln}\left(\begin{array}{cccc}v_1&w_1&v_2&w_2\\v_1&w_1&v_2&w_2\end{array}\right)(-1)}\\&&-x_1x_2\sqrt{\det K^{-1}_{m\times 2ln}\left(\begin{array}{cc}v_1&w_1\\v_1&w_1\end{array}\right)(-1)\cdot\det K^{-1}_{m\times 2ln}\left(\begin{array}{cc}v_2&w_2\\v_2&w_2\end{array}\right)(-1)}\\
&=&x_1x_2(|K^{-1}_{m\times 2ln}(v_1,w_1)K^{-1}_{m\times
2ln}(v_2,w_2)+K^{-1}_{m\times 2ln}(v_1,w_2)K^{-1}_{m\times
2ln}(w_1,v_2)\\&&-K^{-1}_{m\times 2ln}(v_1,v_2)K^{-1}_{m\times
2ln}(w_1,w_2)|-|K^{-1}_{m\times 2ln}(v_1,w_1)K^{-1}_{m\times
2ln}(v_2,w_2)|)\label{covariance}
\end{eqnarray}
In order to compute the covariance as $l\nearrow\infty$, we only
need to compute the entries of $K^{-1}_{m\times 2ln}(-1)$ as
$l\nearrow\infty$. 

Then we have the following proposition

\begin{proposition}
Consider the dimer model on a Fisher graph, embedded into an
$m\times ln$ cylinder, as illustrated in Figure 4. Let the
circumference of the cylinder go to infinity, i.e.
$l\rightarrow\infty$, and keep the height of the cylinder unchanged.
Consider two edges $e_1$ and $e_2$. As $|e_1-e_2|\nearrow\infty$: if
$P(z)=0$ does not intersect $\mathbb{T}$, the edge-edge correlation
decays exponentially ; if $P(z)=0$ has a node at 1, the edge-edge
correlation tends to a constant.
\end{proposition}
\begin{proof}
The theorem follows from formula (\ref{covariance}), and the
estimates of the entries of inverse Kasteleyn matrix. By (\ref{secondterm}), the second term goes to 0 as $|e_1-e_2|\rightarrow\infty$, the first term is 0, if no zero exists on $\mathbb{T}$, the first term is a nonzero constant if the spectral curve has a real node at 1.
\end{proof}

\begin{example}($1\times2$ cylindrical graph)Assume we have a dimer
model on a Fisher graph embedded into an infinite cylinder of height
1 and period 2. One period of the graph is illustrated in Fig 5.

\begin{figure}[htbp]
  \centering
\scalebox{0.8}[0.8]{\includegraphics{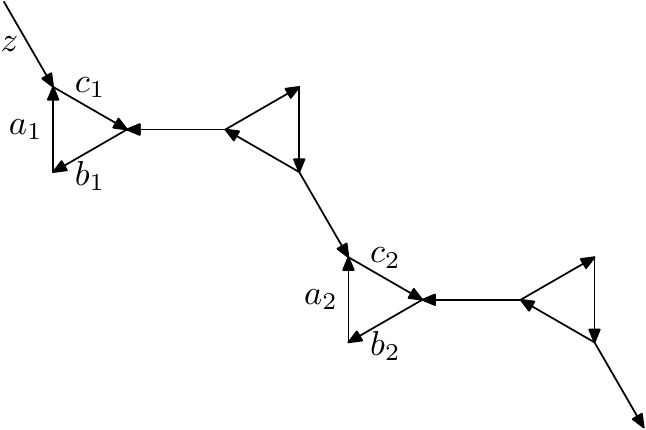}}
   \caption{$1\times 2$ cylindrical graph }
\end{figure}

The characteristic polynomial of the model is
\begin{eqnarray*}
P(z)=(b_1b_2z-a_1a_2)(\frac{b_1b_2}{z}-a_1a_2).
\end{eqnarray*}
The probability that an $a_1$-edge occurs is
\begin{eqnarray*}
Pr(a_1)&=&\frac{a_1a_2}{4\pi}\left(p.v.\int_{\mathbb{T}}\frac{1}{a_1a_2-b_1b_2z}\frac{dz}{\sqrt{-1}z}+p.v.\int_{\mathbb{T}}\frac{1}{a_1a_2-\frac{b_1b_2}{z}}\frac{dz}{\sqrt{-1}z}\right)\\
&=&\left\{\begin{array}{cc}1&\mathrm{if}\
a_1a_2>b_1b_2\\0&\mathrm{if}\
a_1a_2<b_1b_2\\\frac{1}{2}&\mathrm{if}\
a_1a_2=b_1b_2\end{array}\right.
\end{eqnarray*}
$P(z)=0$ has a real node at $1$ if and only if $a_1a_2=b_1b_2$. At
the critical case, the covariance of an $a_1$ edge and a $b_2$ edge,
as their distance goes to infinity, is
\begin{eqnarray*}
Pr(a_1\&b_2)-Pr(a_1)Pr(b_2)=0-\frac{1}{4}=-\frac{1}{4}
\end{eqnarray*}
\end{example}

\section{Graph on a Torus}

\subsection{Combinatorial and Analytic Properties}

To compute the characteristic polynomial $P(z,w)$ of the Fisher
graph, we give an orientation to edges as illustrated in Figure 1.

\begin{lemma} The characteristic polynomial
P(z,w) for a periodic Fisher graph with period $(m,0)$ and $(0,n)$ is a
Laurent polynomial of the following form:
\begin{eqnarray*}
P(z,w)=\sum_{i,j} P_{ij}(z^i w^j+\frac{1}{w^jz^i})
\end{eqnarray*}
where $(i,j)$ are integral points of the Newton polygon with
vertices $(\pm m,0),(0,\pm n),(\pm m,\mp n)$:
\begin{figure}[htbp]
\centering
\includegraphics{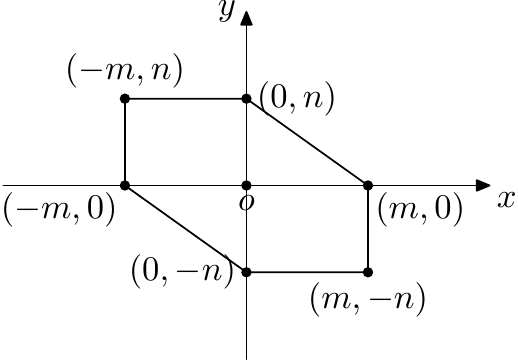}
\caption{Newton polygon}
\end{figure}
\end{lemma}
\begin{proof}
Let $p=6mn$. By definition, \[P(z,w)=\det
K(z,w)=\sum_{i_1,...,i_p}(-1)^{\sigma
(i_1,...,i_p)}k_{1,i_1}\times\cdots\times k_{p,i_p}.\] Here $\sigma$ is
the number of even cycles of the permutation $(i_1,...,i_p)$.  The
sum is over all possible permutations of $p$ elements.  Each term of
$P(z,w)$ corresponds to an oriented loop configuration occupying
each vertex exactly twice. For the graph $G_1$ with $m$ $z$-edges
and $n$ $w$-edges, $P(z,w)$ is a Laurent polynomial with leading
terms
$z^m$,$\frac{1}{z^m}$,$w^n$,$\frac{1}{w^n}$,$\frac{z^m}{w^n}$,$\frac{w^n}{z^m}$.
$z^iw^j$ corresponds to loops of homology class $(i,j)$. $P(z,w)$ is
symmetric with respect to $z^iw^j$ and $\frac{1}{w^jz^i}$. That is
because for each term of $z^iw^j$, if we reverse the orientation of
all loops, we get a term of $\frac{1}{w^jz^i}$ with coefficients of
the same absolute value, corresponding to the product of weights of
edges included in the configuration. The sign of the term is
multiplied by $(-1)^p=1$.

To show that all the powers $(i,j)$ lie in the polygon, we multiply
all the $b$-edges by $z$ (or $\frac{1}{z}$), and all the $c$-edges
by $w$(or $\frac{1}{w}$), according to their orientation. This way
the corresponding characteristic polynomial becomes $P(z^n,w^m)$.
Let $(\tilde{i},\tilde{j})$ be a power of monomial in $P(z^n,w^m)$.
At each triangle, all the possible contributions of local
configurations to the power of the monomial can only be
$(0,0),(0,\pm1),(\pm1,0),(\pm1,\mp 1)$. Examples are illustrated in
the following Figure 4.The left graph has two doubled edges, and the contribution to the power of the monomial is $(0,0)$, the right graph has a loop winding from the $z$-edge to the $w$-edge, and the contribution to the power of the monomial is $(1,-1)$. 
\begin{figure}[htbp]
  \centering
\includegraphics*{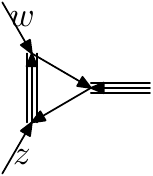}\qquad  \includegraphics*{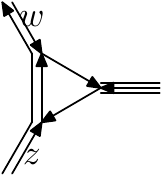}
   \caption{local configurations }
\end{figure}

Consider all the $2mn$ triangles. For each edge, we considered its
contribution twice, since we counted it from both triangles it
connected. Hence we have
\begin{eqnarray*}
-2mn\leq 2\tilde{i}\leq 2mn\\
-2mn\leq 2\tilde{j}\leq 2mn\\
-2mn\leq 2(\tilde{i}+\tilde{j})\leq 2mn
\end{eqnarray*}
Since $\tilde{i}=ni$, $\tilde{j}=mj$, and the Newton polygon $N(P)$
is defined to be
\begin{eqnarray*}
N(P)=\mathrm{convex\ hull}\{(i,j)\in \mathbb{Z}^2|z^iw^j\
\mathrm{is\ a\ monomial\ in\ P(z,w) }\},
\end{eqnarray*}
the lemma follows
\end{proof}

\begin{lemma}
For configurations with odd loops corresponding to a non-vanishing
term in $P(z,w)$, all odd loops have non-trivial homology, and the total number
of odd loops is even.
\end{lemma}
\begin{proof}
For any loop configuration on the Fisher graph embedded on a torus,
the number of odd loops is always even. That is because the total
number of vertices are even, while odd loops always involve odd
number of vertices. Any term in $P_{ij}$ including odd loops can
appear only when odd loops have non-trivial homology. That is because for a
contractible odd loop, we can reverse the orientation of that loop to
negate the sign of that term. The term with reversed orientation on
the odd loop cancels with the original term, because the homology
class $(0,0)$ of the configurations is not changed given the odd
loop is contractible.
\end{proof}

\subsection{Generalized Fisher Correspondence}
Consider the Fisher graph obtained by replacing each vertex of the honeycomb lattice by a triangle. Assume all the triangle edges have weight 1, and all the non-triangle edges have positive weights not equal to 1. Furthermore, we assume that at each triangle, there is an even number of adjacent edges with weight less than 1. We introduce a generalized Fisher correspondence between the Ising model on the triangular lattice and the dimer model on the Fisher lattice as follows:
\begin{enumerate}
\item If two adjacent spins have the same sign, and the dual edge of the Fisher lattice has weight strictly greater than 1, then the dual edge separating the two spins is present in the dimer configuration.
\item If two adjacent spins have the same sign, and the dual edge of the Fisher lattice has weight strictly less than 1, then the dual edge separating the two spins is not present in the dimer configuration.
\item If two adjacent spins have the opposite sign, and the dual edge of the Fisher lattice has weight strictly greater than 1, then the dual edge separating the two spins is not present in the dimer configuration.
\item If two adjacent spins have the opposite sign, and the dual edge of the Fisher lattice has weight strictly less than 1, then the dual edge separating the two spins is present in the dimer configuration. If at each triangle, an even number of incident edges have weight less than 1, we change the configuration on an even number of incident edges. As a result, the number of present edges incident to each triangle is still odd, which is a dimer configuration. This correspondence is 2-to-1 since negating the spins at all vertices corresponds to the same dimer configuration.
\end{enumerate}

Around each triangle of the triangular lattice, we always have an even number of sign changes. If all the non-triangular edges have weight strictly greater than 1, then the number of present edges incident to each triangle of the Fisher graph is equal to the number of edges separating the same spins, which is odd. Whenever we have an edge with weight strictly less than 1, we change the configuration of that edge, according to the principle described above.  Figure 9 is an example of the generalized Fisher correspondence given $a>1$, $b>1$, $c<1$, $d>1$, $e<1$.


\begin{figure}[htbp]
  \centering
\scalebox{0.6}[0.6]{\includegraphics{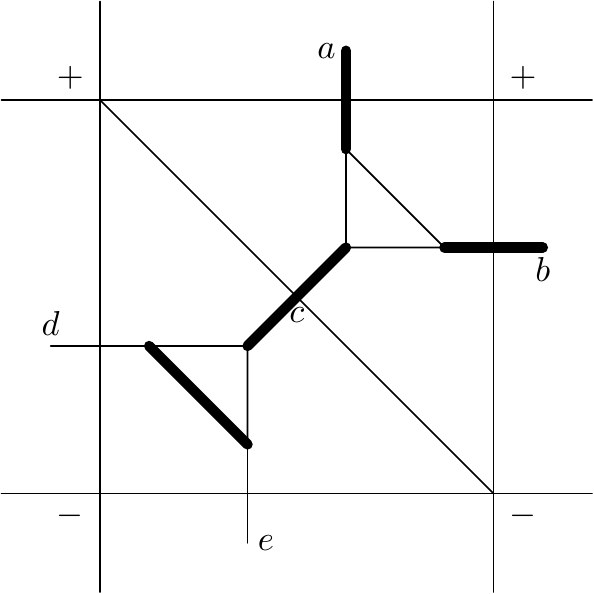}}
   \caption{Generalized Fisher Correspondence}
\end{figure}

Choose the interaction $J_e$ associated to a bond as follows:
\begin{eqnarray*}
J_e=\frac{1}{2}|\log w_e|
\end{eqnarray*}
where $w_e$ is the weight of the dual edge. $J>0$ corresponds to the ferromagnetic interaction.

\subsection{Duality Transformation}
We call a geometric figure built with a certain number of bonds of a
lattice a closed polygon if at every lattice point only an even
number of bonds occurs. It is clear that  every configuration of
``$+$'' and ``$-$'' spins on a lattice can be associated a closed
polygon of the dual lattice in the following way. A dual bond
belongs to the polygon if it separates different spins and does not
belong to the polygon if it separates equal spins. The same closed
polygon is associated to two symmetric configurations in which the
``$+$'' and ``$-$'' spins are interchanged.

Let $T_{mn}$ be the quotient graph of the triangular lattice on the plane, as
defined on Page 4. Let $H_{mn}$ be the dual graph of $T_{mn}$, $H_{mn}$ is a honeycomb lattice which can be embedded into an $m\times n$ torus.  Without loss of generality, assume both $m$ and $n$ are even.

Define an
Ising model on $T_{mn}$ with interactions $\{J_e\}_{e\in E(T_{mn})}$.
Assume the Ising model on $T_{mn}$ has partition function $Z_{T_{mn}, I}$.
Then $Z_{T_{mn},I}$ can be written as, 
\[Z_{T_{mn},I}=2\prod_{e\in E(T_{mn})}\exp(J_e)\sum_{C^*\in S_{00}^*}\prod_{e\in C^*}\exp(-2J_e):=2\prod_{e\in E(T_{mn})}\exp(-J_e)Z_{F_{mn},D_{00}},\]
where $S_{00}^*$ is the set of closed polygon configurations of
$H_{mn}$, with an even number of occupied bonds crossed by both
$\gamma_x$ and $\gamma_y$. The sum is over all configurations in
$S_{00}^*$. Similarly, we can define $S_{01^*}(S_{10}^*,S_{11}^*)$
to be the set of closed polygon configurations of $H_{mn}$, with an
even(odd,odd) number of occupied bonds crossed by $\gamma_x$, and an
odd(even,odd) number of occupied bonds crossed by $\gamma_y$. 

$F_{mn}$ is the Fisher graph obtained by replacing each vertex of $H_{mn}$ by a triangle, with weights on all the non-triangle edges given by
\[w_e=e^{2J_e}\]
Let $Z_{F_{mn}, D}$ be the partition function of dimer configurations on
$F_{mn}$, a Fisher graph embedded into an $m\times n$ torus, with
weights $e^{2J_e}$ on edges of $H_{mn}$, and weight 1 on all the
other edges. Then
\[Z_{F_{mn},D}=Z_{F_{mn},D_{00}}+Z_{F_{mn}, D_{01}}+Z_{F_{mn},D_{10}}+Z_{F_{mn},D_{11}},\]
where
\[Z_{F_{mn},D_{\theta,\tau}}=\prod_{e\in E(T_{mn})}\exp(2J_e)\sum_{C^*\in S_{\theta,\tau}^*}\prod_{e\in C^{*}}\exp(-2J_e).\]
For example, $Z_{F_{mn},D_{01}}$ is the dimer partition function on
$F_{mn}$ with an even number of occupied edges crossed by $\gamma_x$,
and an odd number of occupied edges crossed by $\gamma_y$. It also
corresponds to an Ising model which has the same configuration on
the two boundaries parallel to $\gamma_y$, and the opposite
configurations on the two boundaries parallel to $\gamma_x$. Similar
results hold for all the $Z_{F_{mn}, D_{\theta,\tau}}$,
$\theta,\tau\in\{0,1\}$.

On the other hand, if we consider the high temperature expansion of
the Ising model on $T_{mn}$, we have
\begin{eqnarray*}
Z_{T_{mn},I}&=&\sum_{\sigma}\prod_{e=uv\in
E(T_{mn})}\exp(J_e\sigma_u\sigma_v)\\
&=&\sum_{\sigma}\prod_{e=uv\in E(T_{mn})}(\cosh
J_e+\sigma_u\sigma_v\sinh J_e)\\
&=&\prod_{e=uv\in E(T_{mn})}\cosh J_e\sum_{\sigma}\prod_{e=uv\in
E(T_{mn})}(1+\sigma_u\sigma_v\tanh J_e)\\
&=&\prod_{e=uv\in E(T_{mn})}\cosh J_e\sum_{C\in S}\prod_{e\in
C}2^{mn}\tanh J_e,
\end{eqnarray*}
where $S$ is the set of all closed polygon configurations of $T_{mn}$.
Let $\tilde {F}_{mn}$ be a Fisher graph embedded into an $m\times n$
torus, with weights $\tanh J_e$ on edges of $G_n$, with each vertex of the triangular lattice 
replaced by a gadget, as illustrated in the following Figure.
\begin{figure}[htbp]
  \centering
\includegraphics{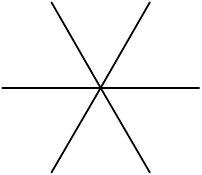}\qquad\includegraphics{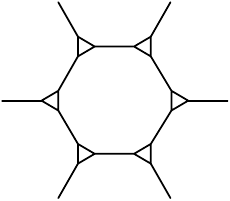}
   \caption{Fisher Correspondence}
\end{figure}

There is an 1-to-2 correspondence between closed polygon configurations on the left graph and the dimer configurations on the right graph. An edge is present in the closed polygon configuration of the left graph if and only if it is present in the dimer configuration of the right graph.
Assume $\tilde{F}_{mn}$ has weight $\tanh J_e$ on edges of $T_{mn}$, and weight 1 on
all the other edges. In other words, the edge with weight $\tanh
J_e$ of $\tilde{F}_{mn}$ and the edge with weight $e^{2J_e}$ of $F_{mn}$
are dual edges. Then we have
\[Z_{T_{mn},I}=2\prod_{e\in E(T_{mn})}\exp(-J_e)Z_{T_{mn},D_{00}}=\prod_{e\in E(T_{mn})}\cosh J_e Z_{\tilde {F}_{mn},D}\]
where $Z_{\tilde{F}_{mn},D}$ is the partition function
of dimer configurations on $\tilde{F}_{mn}$. Hence we have
\[Z_{T_{mn},D_{00}}=\frac{1}{2^{mn+1}}\prod_{e\in
E_{T_{mn}}}(1+\exp(2J_e))Z_{\tilde{F}_{mn},D}.\label{duality}\]

More generally, we can expand all the $Z_{F_{mn},D_{\theta\tau}}$ as
follows:
\begin{equation*}
Z_{F_{mn},D_{\theta,\tau}}=\frac{1}{2^{mn+1}}\prod_{e\in
E_{T_{mn}}}(1+\exp(2J_e))Z_{\tilde{F}_n,D}((-1)^{\tau},(-1)^{\theta}).
\end{equation*}
$Z_{\tilde{F}_{n,D}}(-1,1)$ is the dimer partition function of
$\tilde{F}_n$ with weights of edges crossed by $\gamma_x$ multiplied
by $-1$. Similarly for $Z_{\tilde{F}_{{mn},D}}(1,-1)$ and
$Z_{\tilde{F}_{{mn},D}}(-1,-1)$.
Therefore
\begin{eqnarray*}
Z_{F_{mn},D_{00}}=\max_{\theta,\tau\in\{0,1\}}Z_{F_{mn},D_{\theta,\tau}}
\end{eqnarray*}

Now we consider a Fisher graph $\hat{F}_{mn}$. $\hat{F}_{mn}$ is the same graph as $F_{mn}$ except edge weights.  Namely, $\hat{mn}$ has weight 1 on all the triangle edges. Around each triangle, we have an even number of connecting edges satisfying 
\[\hat{w_e}=\frac{1}{w_e},\]
where $w_e$(resp.\ $\hat{w}_e$) is the weight of edge $e$ for the graph $F_{mn}$(resp.\ $\hat{w}_e$). All the other edge weights satisfy
\[\hat{w_e}=w_e.\]

Without loss of generality, we assume that an even number of edges with weight strictly less than 1 are crossed by $\gamma_x$, and an even number of edges with weight greater than 1 are crossed by $\gamma_y$. Then there is a 1-to-1 correspondence between configurations in $Z_{F_{mn},D_{\theta\tau}}$ and $Z_{\hat{F}_{mn}, D_{\theta\tau}}$ by changing the configurations on all the edges with weight strictly less than 1. Hence we have
\begin{eqnarray*}
Z_{F_{mn},D_{\theta,\tau}}=\prod_{\{e:w_e<1\}}w_e Z_{\hat{F}_{mn},D_{\theta,\tau}}.
\end{eqnarray*}
As a result
\begin{eqnarray}
Z_{\hat{F}_{mn},D_{00}}=\max_{\theta,\tau\in\{0,1\}}Z_{\hat{F}_{mn},D_{\theta,\tau}}\label{maximalpartition}
\end{eqnarray}

\begin{proposition}Assume all the triangle edges have weight 1, and all the non-triangle edges have weight not equal to 1. Assume around each triangle, an even number of edges have weight strictly less than 1. Assume the size of the graph $m$ and $n$ are even, and the number of edges crossed by $\gamma_x$ and $\gamma_y$ with weight strictly less than 1 are both even.  Then $P(z,-1)=0$ have no zeros on the unit circle $\mathbb{T}$. 
\end{proposition}

\begin{proof} When both $m$ and $n$ are even, we have
\begin{eqnarray*}
Pf K(1,1)&=&Z_{00}-Z_{01}-Z_{10}-Z_{11}\\
Pf K(1,-1)&=&Z_{00}+Z_{01}-Z_{10}+Z_{11}\\
Pf K(-1,1)&=&Z_{00}-Z_{01}+Z_{10}+Z_{11}\\
Pf K(-1,-1)&=&Z_{00}+Z_{01}+Z_{10}-Z_{11}
\end{eqnarray*}
By (\ref{maximalpartition}), $Pf K(1,-1)>0$, $Pf K(-1,1)>0$, $Pf K(-1,-1)>0$, given all the edge weights are strictly positive. Hence $P(z,-1)$ have no real roots on $\mathbb{T}$.

  Assume
$P(z)=0$ has a non-real zeros $z_0\in\mathbb{T}^2$. Assume
\begin{eqnarray*}
z_0=e^{\sqrt{-1}\alpha_0\pi},\qquad \alpha_0\in(0,1)\cup(1,2)
\end{eqnarray*}
If $\alpha_0$ is rational, namely $a_0=\frac{p}{q}$, then after enlarging the fundamental domain to $m\times qn$, $P_{q}(z^q,-1)=0$, while $z^q$ is real, which is a impossible.

 Now let us consider the case of irrational $\alpha_0$. There exists a sequence
 $\ell_k\in\mathbb{N}$, such that
 \begin{eqnarray*}
 \lim_{k\rightarrow\infty}z_0^{\ell_k}=1
 \end{eqnarray*}
 In other words, if we assume
 $z_0^{\ell_k}=e^{\sqrt{-1}\alpha_k\pi}$ where
 $\alpha_k\in(-1,1)$, then
 \begin{eqnarray*}
 \lim_{k\rightarrow\infty}\alpha_k=0.
 \end{eqnarray*}
 According to the formula of enlarging the fundamental domain,
 \[P_{\ell_k}(z_0^{\ell_k},-1)=0\qquad \forall k\]
 
 By (\ref{maximalpartition}),
 \[P(1,-1)=(Z_{00}+Z_{01}-Z_{10}+Z_{11})^2\geq (Z_{01}+Z_{11})^2\]
 Therefore we have
 \begin{eqnarray}
 1&\leq&\lim_{k\rightarrow\infty}\left|\frac{P_{\ell_k}(z_0^{\ell_k},-1)-P_{\ell_k}(1,-1)}{(Z_{\ell_k,01}+Z_{\ell_k,11})^2}\right|\label{biginequality}
 \end{eqnarray}
 On the other hand
 \begin{eqnarray}
&& \lim_{k\rightarrow\infty}\left|\frac{P_{\ell_k}(z_0^{\ell_k},-1)-P_{\ell_k}(1,-1)}{(Z_{\ell_k,01}+Z_{\ell_k,11})^2}\right|\\
 &=&\lim_{k\rightarrow\infty}\frac{1}{(Z_{\ell_k,01}+Z_{\ell_k,11})^2}|P_{\ell_k}(e^{\sqrt{-1}\alpha_k\pi})-P_{\ell_k}(e^{\sqrt{-1}\pi})|\\
 &=&\lim_{k\rightarrow\infty}\frac{1}{(Z_{\ell_k,01}+Z_{\ell_k,11})^2}\left|\sum_{t=1}^{\infty}\frac{[\pi\alpha_k]^{2t}}{(2t)!}\frac{\partial^{2t}P_{\ell_k}(e^{\sqrt{-1}\theta})}{\partial\theta^{2t}}|_{\theta=0}\right|\label{approximation}
 \end{eqnarray}
Since
\begin{eqnarray*}
P_{\ell_k}(z,-1)=\sum_{0\leq j\leq m}P_{j}^{(\ell_l)}(z_j+\frac{1}{z^j})
\end{eqnarray*}
where $P_j^{(\ell_k)}$ is the signed sum of loop configurations winding exactly $j$ times around the cylinder (see Lemma 2.1), we have
\begin{eqnarray*}
\frac{\partial^{2t}P_{\ell_k}(e^{\sqrt{-1}\theta})}{\partial\theta^{2t}}|_{\theta=0}=\sum_{1\leq j\leq m}2j^{2t}(-1)^tP_j^{(\ell_k)},
\end{eqnarray*}
and
\begin{eqnarray*}
\lim_{k\rightarrow\infty}\frac{1}{(Z_{\ell_k,01}+Z_{\ell_k,11})^2}\left|\sum_{t=1}^{\infty}\frac{(\pi\alpha_k)^{2t}}{(2t)!}\frac{\partial^{2t}P_{\ell_k}(e^{\sqrt{-1}\theta})}{\partial\theta^{2t}}|_{\theta=0}\right|\\
\leq\lim_{k\rightarrow\infty}\frac{1}{(Z_{\ell_k,01}+Z_{\ell_k,11})^2}\sum_{t=1}^{\infty}\frac{(\pi m\alpha_k)^{2t}}{(2t)!}2\sum_{1\leq j\leq m}|P_j^{(\ell_k)}|
\end{eqnarray*}

Using the same technique as described in the proof of Theorem 1.2, we have
\begin{eqnarray*}
\sum_{1\leq j\leq m}|P_j^{(\ell_k)}|\leq \mathrm{Partition\ of\ configurations\ including\ nonplanar\ odd\ loops}\\+\mathrm{Partition\ of\ configurations\ with\ only\ even\ loops}\\
\leq(2^mC_1^{6m}+1)\mathrm{Partition\ of\ even\ loop\ configurations}\leq C_2^mZ_{\ell_k}^2
\end{eqnarray*}

Moreover, there is a one-to-one correspondence between dimer configurations in $Z_{00}$ and $Z_{01}$, similarly between dimer configurations in $Z_{10}$ and $Z_{11}$. We divide the $m\times \ell_k n$ torus in to $\ell_k n$ circles, each circle has circumference $m$. Fix one circle $\mathcal{C}_0$, and fix configurations out side the circle and on the boundary of the circle. For each dimer configuration in $Z_{00}$, if we rotate the configuration to alternating edges along $C_0$, we get a dimer configuration in $Z_{01}$. An example of such an transformation is illustrated in the Figures 11 and 12.

\begin{figure}[htbp]
\centering
\scalebox{0.8}[0.8]{\includegraphics{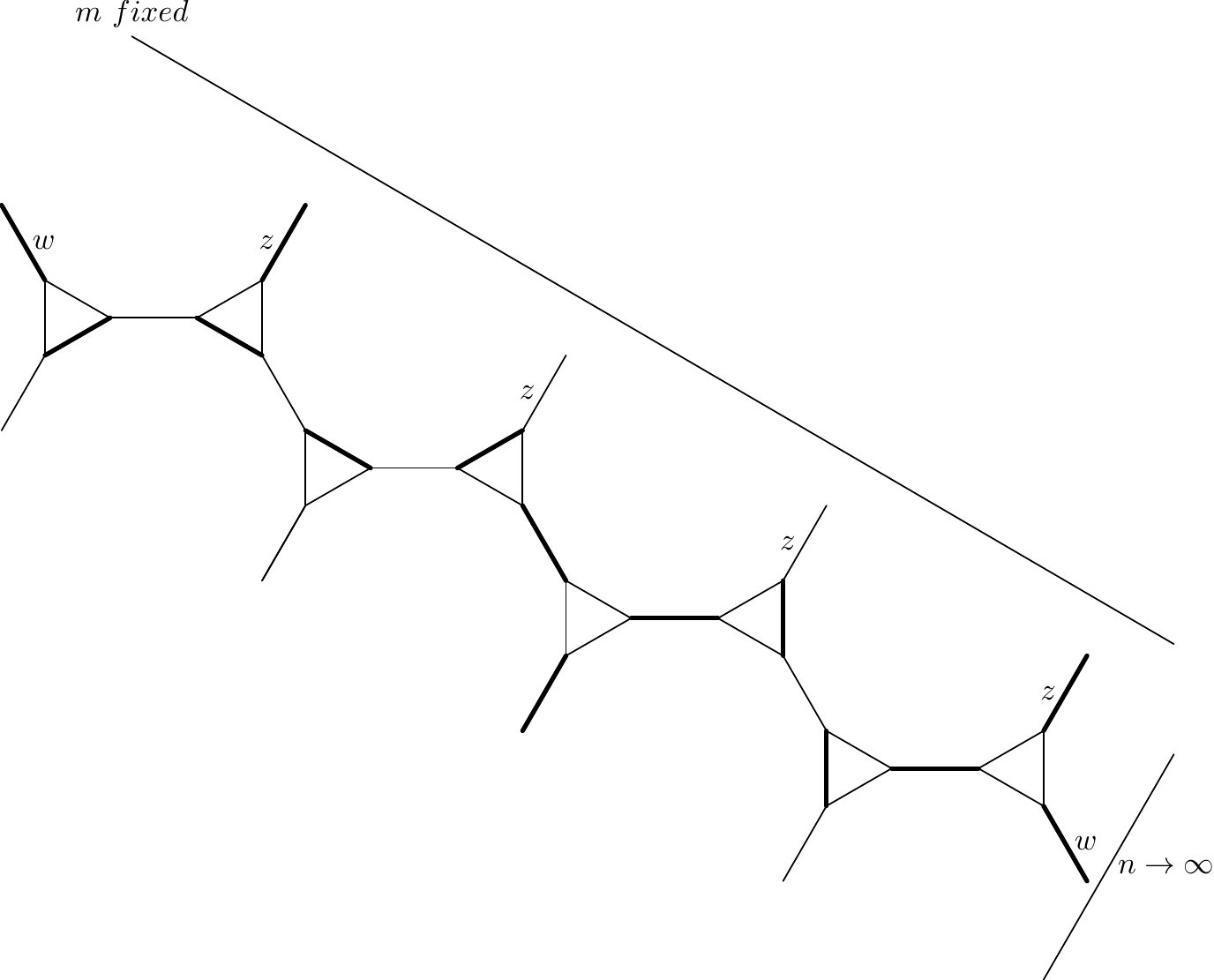}}
\caption{}
\end{figure}

\begin{figure}[htbp]
\centering
\scalebox{0.8}[0.8]{\includegraphics{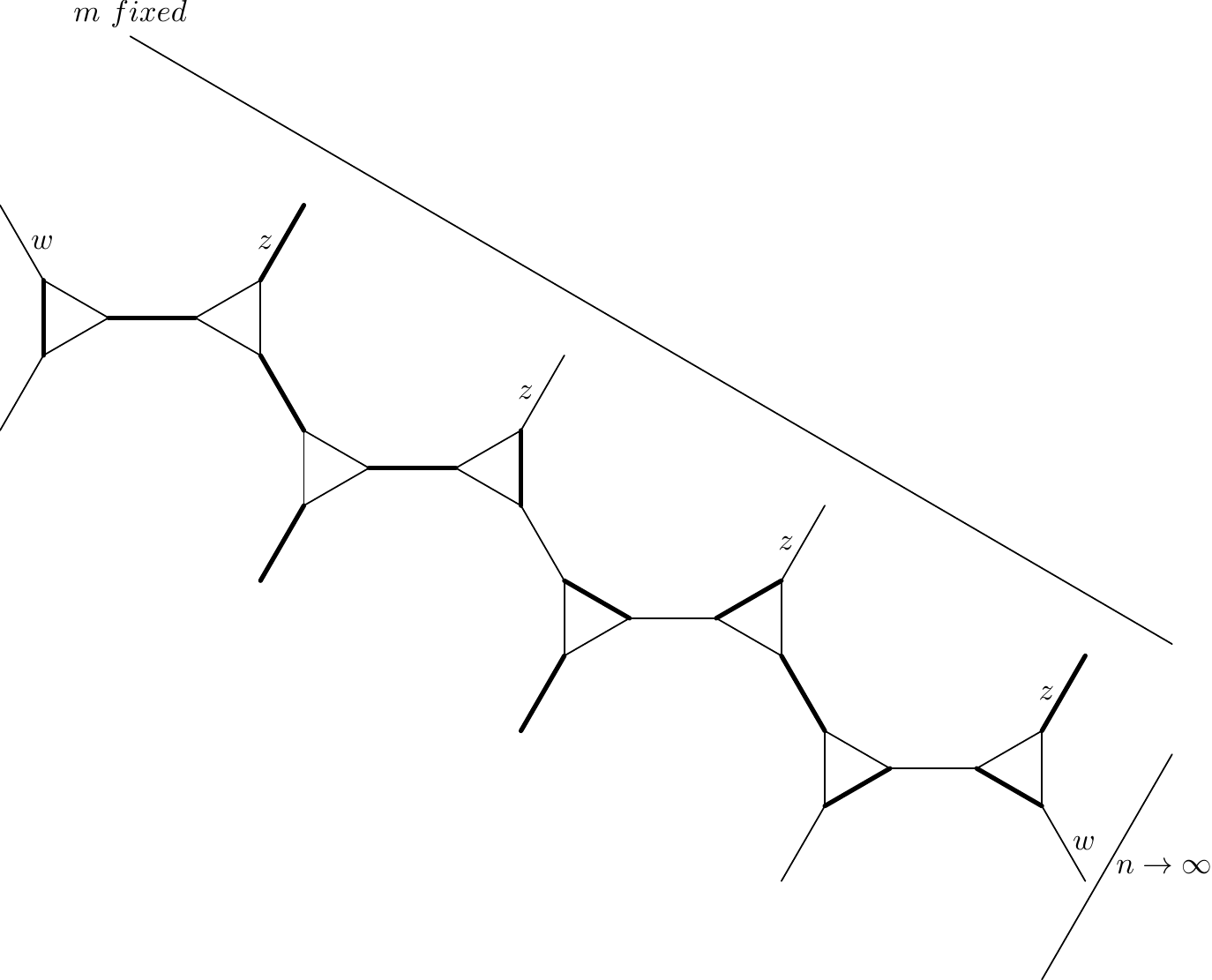}}
\caption{}
\end{figure}

As a result,
\begin{eqnarray*}
(Z_{\ell_k,00}+Z_{\ell_k,01}-Z_{\ell_k,10}+Z_{\ell_k,11})^2\leq C^m (Z_{\ell_k,01}+Z_{\ell_k,11})^2
\end{eqnarray*}
where $C$ is a constant independent of $k$. Hence
\begin{eqnarray*}
&&\lim_{k\rightarrow\infty}\frac{1}{(Z_{\ell_k,01}+Z_{\ell_k,11})^2}\left|\sum_{t=1}^{\infty}\frac{(\pi\alpha_k)^{2t}}{(2t)!}\frac{\partial^{2t}P_k(e^{\sqrt{-1}\theta})}{\partial\theta^{2t}}|_{\theta+0}\right|\\
&\leq&\lim_{k\rightarrow\infty}\frac{1}{(Z_{\ell_k,01}+Z_{\ell_k,11})^2}\sum_{t=1}^{\infty}\frac{(m\pi\alpha_k)^{2t}}{(2t)!}C_2^mZ_{\ell_k}^2\\
&\leq&\lim_{k\rightarrow\infty}C_3^m\sum_{t=1}^{\infty}\frac{(m\pi\alpha_k)^{2t}}{(2t)!}
\end{eqnarray*}
Since $m$ is a constant, and $\lim_{k\rightarrow\infty}\alpha_k=0$, we have
\begin{eqnarray*}
\lim_{k\rightarrow\infty}\sum_{t=1}^{\infty}\frac{(m\pi\alpha_k)^{2t}}{(2t)!}C_3^m=0
\end{eqnarray*}
which is a contradiction to (\ref{approximation}). Therefore the proposition follows.

\end{proof}

\begin{lemma}$P(z,w)\geq 0,\qquad\forall (z,w)\in\mathbb{T}^2$.
\end{lemma}
\begin{proof}Assume there exists $(z_0,w_0)\in\mathbb{T}^2$, such that $P(z_0,w_0)<0$, since $P(z,-1)>0$ and $P(-1,w>0)$, if we consider $z=e^{i\theta},w=e^{i\phi}$, $z_0=e^{i\theta_0}$, $w_0=e^{i\phi_0}$, and $(\theta,\phi)\in[-\pi,\pi]^2$, on the boundary of the domain $[-\pi,\pi]^2$, $P(e^{i\theta},e^{i\phi})$ are strictly positive, while $P(e^{i\theta_0},e^{i\phi_0})<0$, by continuity there exists a neighborhood $O_h=(\phi_0-h,\phi_0+h)$, such that for any , such that for any $\phi\in O_h$, $P(e^{i\theta_0},e^{i\phi})<0$. Consider  straight lines in $[-\pi,\pi]^2$, connecting $(\theta_0,\phi)$ to $(\pi,\phi)$, namely
\begin{eqnarray*}
\gamma_{\phi}(t)=((\pi-\theta_0)t+\theta_0,\phi)
\end{eqnarray*}
for $\phi\in(\phi_0-h,\phi_0+h)$, there exists $t_{\phi}$, such that
\begin{eqnarray*}
P(e^{i[(\pi-\theta_0)t_{\phi}+\theta_0]},e^{i\phi})=0
\end{eqnarray*}
By continuity, there exists a $\phi$ in the open interval $(\phi_0-h,\phi_0+h)$ such that $\frac{\phi}{\pi}$ a  of -1. For instance $\phi=\frac{q}{p}\pi$ in reduced form where $q$ is odd. Then after enlarging the fundamental domain to $p m\times n$, we have $P(e^{i[(\pi-\theta_0)t_{\phi}+\theta_0]},-1)=0$, which is a contradiction.

\end{proof}

\subsection{Proof of Theorem 1.2}

In this section, we prove that the spectral curve of Fisher graphs associated to any periodic, ferromagnetic Ising model is a Harnack curve, and its intersection with the unit torus $\mathbb{T}^2=\{(z,w):|z|=1,|w|=1\}$ is either empty or a single real intersection of multiplicity 2.

First of all, it is proved in \cite{d} that the Ising characteristic polynomial  $P(z,w)$ has the same zero locus as the characteristic polynomial of the square-octagon lattice, denoted by $P_C(z,w)$, with suitable clock-wise odd orientation. $P_{C}(z,w)$ is the characteristic polynomial of a bipartite graph with positive edge weights and clockwise-odd orientation, it is proved in \cite{ko} that $P_{C}(z,w)=0$ is a Harnack curve, i.e, for any fixed $x,y>0$ the number of zeros of $P_{C}(z,w)$ on the torus $\mathbb{T}_{x,y}=\{(z,w):|z|=x,|w|=y\}$ is at most 2 (counting multiplicities). Hence the number of zeros of $P_{\mathcal{F}}(z,w)$ on the unit torus is at most 2. Since $P_{\mathcal{F}}(z,w)\geq 0$, $\forall (z,w)\in \mathbb{T}^2$, each zero $P_{\mathcal{F}}(z,w)$ on $\mathbb{T}^2$ is at least of multiplicity 2. If $P_{\mathcal{F}}(z,w)$ have a non-real zero $(z_0,w_0)$  on $\mathbb{T}^2$, then $(\bar{z}_0,\bar{w}_0)\neq (z_0,w_0)$ is also a zero of $P_{\mathcal{f}}(z,w)$, and each of them is at least of multiplicity 2. Then $P_{\mathcal{F}}(z,w)$ has at least 4 zeros on $\mathbb{T}^2$, which is a contradiction to the fact that $P_{\mathcal{F}}(z,w)=0$ is a Harnack curve. Hence the intersection of $P_{\mathcal{F}}(z,w)=0$ with $\mathbb{T}^2$ is either empty or a single real intersection of multiplicity 2.

\begin{example}($1\times n$ fundamental domain)
Assume the Fisher graph has period $1\times n$, and all edge weights
are strictly positive. The intersection of $P(z,w)=0$ with
$\mathbb{T}^2$ is either empty or a single real node. That is, they
intersect at one of the point $(\pm1,\pm1)$ and the intersection is
of multiplicity 2.
\end{example}
\begin{proof}
Without loss of generality, assume $P(z,w)$ is linear with respect
to $z$ and $\frac{1}{z}$, then
all terms in $P(z,w)$ fall into two categories\\
$\rmnum{1})$ occupy $z$-edge exactly once\\
$\rmnum{2})$ does not occupy $z$-edge or occupy $z$-edge twice\\
$1\times n$ fundamental domain is comprised of $n\ 1\times 1$
blocks, for each block, we give weights
$a_{i1},b_{i1},c_{i1},a_{i2},b_{i2},c_{i2}$ for triangle edges,
assume $a_{i}=a_{i1}a_{i2}$, $b_{i}=b_{i1}b_{i2}$,
$c_{i}=c_{i1}c_{i2}$, $1\leq i\leq n$. See Figure 6. Obviously such
an arrangement of edge weights is gauge equivalent to the
arrangement giving all triangle edges weight 1 and non-triangle
edges positive edge weights.

\begin{figure}[htbp]
  \centering
\scalebox{0.8}[0.8]{\includegraphics{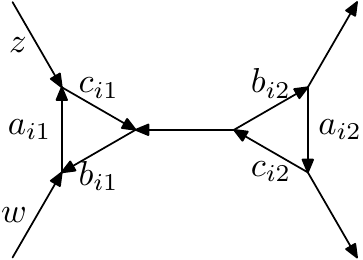}}
   \caption{Fisher Graph on a Cylinder}
\end{figure}

Each configuration in case $\rmnum{1}$) consists of a single even
loop and some double edges, the partition of all configurations in
$\rmnum{1}$ is
\begin{eqnarray*}
P_1=-\prod_{i=1}^{n}[(a_i-b_ic_i)w+(c_i-a_ib_i)]\frac{1}{z}-\prod_{i=1}^{n}[(a_i-b_ic_i)\frac{1}{w}+(c_i-a_ib_i)]z
\end{eqnarray*}
Configurations in case $\rmnum{2}$) depend on configurations of each
$1\times 1$ block, which are determined by configurations of
boundary edges. Let $V_{i}^{0,0}$ denote the partition at $i-th$
block when both edges in $z$ direction are unoccupied, $V_{i}^{2,0}$
denote the partition at $i-th$ block when left edge in $z$ direction
is occupied twice, while right edge is unoccupied, similarly for
$V_{i}^{0,2}$ and $V_{i}^{2,2}$. Then we have
\begin{eqnarray*}
V_{i}^{0,0}&=&a_i^2+c_i^2+a_ic_i(w+\frac{1}{w})\\
V_{i}^{0,2}&=&a_{i1}c_{i1}b_{i2}(\frac{1}{w}-w)\\
V_{i}^{2,0}&=&b_{i1}a_{i2}c_{i2}(w-\frac{1}{w})\\
V_{i}^{2,2}&=&b_{i}^2-b_{i}(w+\frac{1}{w})+1
\end{eqnarray*}
Let $k_{i}\in \{0,2\}$, $1\leq i\leq n$, then partition of all
configurations in $\rmnum{2}$ is
\begin{eqnarray*}
P_{0,2}=\sum_{k_1,...,k_n}\prod_{i=1}^{n}V_{i}^{k_i,k_{i+1}}
\end{eqnarray*}
Assume $w=e^{\sqrt{-1}\phi}$, then $V_{i}^{0,0}\geq 0$,
$V_{i}^{2,2}\geq 0$. Periodicity implies that $V_i^{0,2}$ and
$V_{i}^{2,0}$ always appear in pairs, then all terms in $P_{0,2}$
where $\sin\phi$ appears depend only on $\sin^2\phi$, moreover,
given all edge weights are positive, each term with $\sin\phi$ is
nonnegative, therefore
\begin{eqnarray*}
P(z,w)&=&P_1+P_{0,2}\\
&=&-\prod_{i=1}^{n}[(a_i-b_ic_i)w+(c_i-a_ib_i)]\frac{1}{z}-\prod_{i=1}^{n}[(a_i-b_ic_i)\frac{1}{w}+(c_i-a_ib_i)]z\\
&&+\prod_{i=1}^{n}[a_i^2+c_i^2+a_ic_i(w+\frac{1}{w})]+\prod_{i=1}^{n}[b_{i}^2-b_{i}(w+\frac{1}{w})+1]+F(w)
\end{eqnarray*}
where $F(w)\geq 0$. Assume $G(z,w)=z[P(z,w)-F(w)]$. Fix $w$,
$G(z,w)$ is a quadratic polynomial with respect to $z$. Assume
\begin{eqnarray*}
C_{0}&=&\prod_{i=1}^{n}[(a_i-b_ic_i)\frac{1}{w}+(c_i-a_ib_i)]\\
C_{1}&=&\prod_{i=1}^{n}[a_i^2+c_i^2+a_ic_i(w+\frac{1}{w})]+\prod_{i=1}^{n}[b_{i}^2-b_{i}(w+\frac{1}{w})+1]
\end{eqnarray*}
then the root for $G(z,w)=0$ will be
\begin{eqnarray*}
z_{1,2}=\frac{-C_1\pm(C_1^2-4|C_0|^2)^{\frac{1}{2}}}{2C_0}
\end{eqnarray*}
On the other hand,
\begin{eqnarray*}
C_1^{2}-4|C_0|^2\geq
4[\prod_{i=1}^{n}(a_i^2+c_i^2+2a_ic_i\cos\phi)(b_i^2+1-2b_i\cos\phi)\\
-\prod_{i=1}^{n}[(a_i-b_ic_i)^2+(c_i-a_ib_i)^2+2(a_i-b_ic_i)(c_i-a_ib_i)\cos\phi]]
\end{eqnarray*}
\begin{eqnarray*}
(a_i^2+c_i^2+2a_ic_i\cos\phi)(b_i^2+1-2b_i\cos\phi)-[(a_i-b_ic_i)^2+(c_i-a_ib_i)^2+2(a_i-b_ic_i)(c_i-a_ib_i)\cos\phi]]\\
=4a_ib_ic_i\sin^2\phi\geq 0
\end{eqnarray*}
Therefore, $C_1^2-4|C_0|^2\geq 0$. Given $a_i,b_i,c_i>0$, equality
holds only if $w$ is real. If $C_1^2-4|C_0|>0$, then $|z_{1,2}|\neq
1$, $G(z,w)$ has no zeros on $\mathbb{T}^2$. Since $G(1,1)=\det
K(1,1)>0$, $P(z,w)-F(w)>0$ on $\mathbb{T}^2$, so $P(z,w)>0$ on
$\mathbb{T}^2$. If $C_1^2-4|C_0|^2=0$, then $w$ is real, $z_1=z_2$
are real, $P(z,w)$ has real node on $\mathbb{T}^2$.
\end{proof}

\end{document}